\documentclass[12pt]{amsart}
\usepackage{amssymb}
\usepackage{amsmath}
\usepackage{xcolor}
\usepackage{enumitem}
\usepackage{fullpage}
\usepackage{appendix}
\usepackage{float}
\usepackage{array}

\usepackage[unicode,hypertexnames=false]{hyperref}
\usepackage{bookmark}

\hypersetup{
	colorlinks=true,
	linkcolor=blue!60!black,
	citecolor=blue!60!black,
	urlcolor=blue!60!black
}

\allowdisplaybreaks
\newtheorem{thm}{Theorem}[section]
\newtheorem{prop}[thm]{Proposition}
\newtheorem{cor}[thm]{Corollary}
\newtheorem{lem}[thm]{Lemma}
\newtheorem{rem}[thm]{Remark}
\newtheorem{defn}[thm]{Definition}

\numberwithin{equation}{section}

\newcommand{\C}{\mathbb{C}}
\newcommand{\N}{\mathbb{N}}
\newcommand{\Z}{\mathbb{Z}}

\newcommand{\Res}{\mbox{\rm Res}}
\newcommand{\Span}{\mbox{\rm Span}}
\newcommand{\End}{\mbox{\rm End}}
\newcommand{\Hom}{\mbox{\rm Hom}}
\newcommand{\wt}{\mathrm{wt} \, }

\title[Generators and relations for $A_n(V)$]{On generators and relations for $O_n^\circ(V)$ in higher-level Zhu algebras, Part II}

\author{Nathan Vander Werf}
\address{Department of Mathematics and Statistics, University of Nebraska at Kearney, Kearney, NE 68849}
\email{vanderwerfnp@unk.edu}
\urladdr{\href{https://orcid.org/0009-0007-7374-9345}{https://orcid.org/0009-0007-7374-9345}}


\date{October 25, 2026}


\begin{document}

\begin{abstract}
We study circle relations in the higher-level Zhu algebra of a vertex
operator algebra $V$. For regular vacuum words in homogeneous letters, we
prove the termwise correction-residue conjecture of Part~I: every correction
in the one-sided iterate expansion belongs separately to the finite-length
plus subspace. The first correction formed with the left truncated field
vanishes identically. Multiplying the diagonal commutator kernel by a
power of $1+x$ turns it into a Laurent polynomial. This proves the
vanishing and writes every other correction as a finite combination of
product residues with fewer $Y_L$ factors; these residues come from the
word obtained by replacing one negative mode by a nonnegative one. These results hold at all levels
$n\ge0$, without an auxiliary-degree hypothesis.

For $n\ge1$ and a homogeneous ordered pair $(u,v)$, two explicit triangular
families span the quadratic residue operators modulo the length-one and
singular-iterate relations, giving
\[
\dim\mathcal Q_n((u,v))\le n(n+1).
\]
The universal product-coefficient arrays have dimension exactly $n(n+1)$
and three natural bases; restriction to the regular mode square is an
isomorphism. This coefficient statement proves operator spanning and gives
invertible matrices for two complementary diagonal eliminations.
For equal inputs, mode interchange reduces the operator families to their
odd-indexed members; its scalar coefficient action is an involution. This yields
\[
\dim\mathcal Q_n((u,u))
\le 2\left\lfloor\frac{(n+1)^2}{4}\right\rfloor.
\]
The pairwise estimates give an explicit global quadratic bound for any
finite homogeneous strong generating set. Heisenberg, Virasoro, and affine
$\mathfrak{sl}_2$ examples illustrate equality and strict inequality in
the operator bounds. The correction theorem and quadratic eliminations
provide the inputs for the higher-length reductions in Part~III.
\end{abstract}

\maketitle
\tableofcontents
\clearpage

\section{Introduction}
\label{sec:introduction}

For a vertex operator algebra $V$, the higher-level Zhu algebra
$A_n(V)=V/O_n(V)$ records the representation theory of $V$ visible at
level $n$ \cite{DLM,Z2}. By definition \cite{DLM}, $O_n(V)$ is spanned by
the level-$n$ circle products and the vectors $(L(-1)+L(0))v$; in the
notation of \cite[(2.2)--(2.3)]{AB},
\[
O_n(V)=O^L(V)+O_n^\circ(V),
\]
where $O^L(V)$ is spanned by these translation relations and
$O_n^\circ(V)$ by the level-$n$ circle products. Explicit presentations
of $A_n(V)$ therefore require a workable description of $O_n^\circ(V)$.
At positive level a single circle relation contains many modes, and
relations attached to iterated vacuum words involve both truncation and
commutator terms.

Explicit computations for the Heisenberg and Virasoro vertex operator
algebras include the level-one algebras of
\cite{BVY-Heisenberg,BVY-Virasoro} and the level-two Heisenberg algebra
of \cite{AB2}. The present paper isolates the circle relations used in
such presentations, before imposing the translation relations in
$O^L(V)$. We use the formal calculus and VOA identities of
\cite[Chs.~2--3]{LL}. For further background on higher-level Zhu algebras
we refer to the introduction of Part~I \cite{Part1}.

Part~I introduced the spaces $O_n^\circ(N,S,V)$ of circle
relations attached to words of length at most $N$ in homogeneous strong
generators $S$ \cite[Definition~4.5]{Part1}. Its left truncated field $Y_L$
\cite[Definition~2.14]{Part1} permits a one-sided
iterate expansion into an ordered product and individual correction
terms \cite[Proposition~3.11]{Part1}. Two questions then arise: do those corrections already belong
to the plus subspace, and how many product relations are needed? We
answer the first question at every word length and the second for
quadratics. The finite-length conjecture of Part~I \cite[Conjecture~4.9]{Part1} remains the subject
of Part~III \cite{Part3}, where its auxiliary-degree hypothesis is used. Neither
the correction theorem nor the quadratic results here requires that
hypothesis.

\subsection*{Correction residues at every word length}

For a regular word
$w=s^{(1)}_{-r_1}\cdots s^{(N)}_{-r_N}\mathbf1$, with $r_j\ge1$,
the plus subspace $O_n^\circ(N,S,V)^+$ contains relations from shorter
words and from words of length $N$ with a nonnegative mode before
the last letter. The one-sided expansion in Part~I has one correction
for each position of $w$ before the last letter;
\cite[Conjecture~4.6]{Part1} asserts membership of each correction
separately in this plus subspace.

\begin{thm}[Termwise correction residues]
\label{thm:main-correction-results}
Let $n\ge0$, let $S$ be a finite homogeneous subset of $V$, and let
$w$ be a regular word of length $N\ge2$ in letters of $S$. Let
$m\ge k\ge0$ and $z\in V$.
\begin{enumerate}[label=\textup{(\roman*)}]
\item The first correction, formed with $Y_L$ as in
\eqref{eq:word-correction-piece} with $t=0$, vanishes:
\[
(wR^n_{m,k}z)_{\mathrm m_0}=0
\qquad\text{in }V.
\]
\item Each correction residue in the one-sided iterate expansion, and each
of its summands with a fixed singular index, belongs separately to the
subspace $O_n^\circ(N,S,V)^+$. Consequently the full residue and its ordered product
representative satisfy
\[
wR^n_{m,k}z\equiv_{O_n^\circ(N,S,V)^+}
(wR^n_{m,k}z)_{\mathrm p}.
\]
\end{enumerate}
\end{thm}

Part~I's conjecture assumes that $S$ strongly generates $V$; the theorem
does not. Part~\textup{(i)} uses only that the letters of $w$ are
homogeneous; for a pair $w=u_{-r}v$ it is the vanishing of the commutator
contribution \eqref{eq:comm-piece-def-mk}. Part~\textup{(i)} is
Corollary~\ref{prop:Rcomm-strip-vanish-main}\textup{(ii)}. With the full
field $Y$, the first correction instead belongs to the corresponding
singular-iterate relation space
(Corollary~\ref{prop:Rcomm-strip-vanish-main}\textup{(iii)}).

The precise correction formulas and the proof of part~\textup{(ii)} are given in
Theorem~\ref{thm:correction-residues-plus}. The proof begins with the
complete diagonal kernel, including its nonnegative-mode difference
tail. Multiplying $F^+_{i,r,u}(x,x)$ by $(1+x)^{r+i}$ produces a
Laurent polynomial supported in degrees $-r-i,\ldots,0$.
This support bound expresses every correction as a finite sum of product
residues of a contracted word, obtained by replacing one negative mode
of $w$ by a nonnegative one, with fewer $Y_L$ factors. For the
correction with $t=0$, in which the first letter $s^{(1)}$ carries the
nonnegative mode, these residues vanish; this gives part~\textup{(i)}.
Contracted words are themselves words in the letters of $S$ with
a nonnegative mode, so an induction on the number of $Y_L$ factors
identifies all errors as relations of shorter words or of singular words
on the original letters. At level zero all the $Y_L$ product and
correction residues vanish.

\subsection*{Quadratic spanning and coefficient bases}

Fix $n\ge1$ and homogeneous $u,v\in V$. The subspace attached to the pair is
\begin{equation}\label{eq:intro-U-def}
U=O_n^\circ(u,V)+O_n^\circ(v,V)
+\sum_{i\ge0}O_n^\circ(u_i v,V).
\end{equation}
The ordered-pair operator space is
\begin{equation}\label{eq:intro-Q-def}
\begin{aligned}
\mathcal Q_n((u,v)):=\Span_\C\bigl\{
&z\longmapsto (u_qv_{-r}\mathbf1)R^n_{m,k}z+U:\\
&q\in\Z,\ r\ge1,\ m\ge k\ge0\bigr\}\subseteq\Hom(V,V/U).
\end{aligned}
\end{equation}
Equality of induced operators means that their difference takes every
input into $U$. After the last-mode reduction of Part~I, the pair
problem reduces to the products $(u_{-r}v)R^n_{m,k}$ modulo $U$.

We use three triangular families throughout; they are restated with
labels as \eqref{eq:A-strip}--\eqref{eq:C-strip} in
Section~\ref{subsec:formal-independence}:
\begin{align*}
\mathrm{(A)}&=
\left\{(u_{-r}v)R^n_{m,0}(\cdot):
n\le m\le2n-1,\ 1\le r\le2n-m\right\},\\
\mathrm{(B)}&=
\left\{(u_{-r}v)R^n_{m,n}(\cdot):
n\le m\le2n-1,\ 1\le r\le2n-m\right\},\\
\mathrm{(C)}&=
\left\{(u_{-r}v)R^n_{m,0}(\cdot):
0\le m\le n-1,\ 1\le r\le n-m\right\}.
\end{align*}
Each family has $\binom{n+1}{2}$ members. Families
$\mathrm{(A)}$ and $\mathrm{(B)}$ give the principal operator spanning
list; $\mathrm{(A)}$ and $\mathrm{(C)}$ give the two diagonal
eliminations used in the higher-length calculations.

\begin{thm}[Quadratic spanning and formal rank]
\label{thm:main-quadratic-results}
Let $n\ge1$ and let $u,v\in V$ be homogeneous.
\begin{enumerate}[label=\textup{(\arabic*)}]
\item The images of $\mathrm{(A)}\cup\mathrm{(B)}$ span
$\mathcal Q_n((u,v))$. Hence
\[
\dim\mathcal Q_n((u,v))\le n(n+1).
\]
\item The span of the complete scalar product-coefficient arrays has
dimension $n(n+1)$. Restriction to the regular mode square
$-2n-1\le a,b\le-1$ identifies it with
$\mathcal S_{\mathrm{prod,reg}}(n;u,v)$. Each of
\[
\mathcal A\cup\mathcal B,\qquad
\mathcal C\cup\mathcal B,\qquad
\mathcal A\cup\mathcal C
\]
is a basis, where $\mathcal A,\mathcal B,\mathcal C$ are the coefficient
vectors of the families \textup{(A)}, \textup{(B)}, \textup{(C)}
\eqref{eq:formal-ABC-seeds}.
\end{enumerate}
\end{thm}

The coefficient statement is Proposition~\ref{prop:AB-prodreg-full-rank};
its all-mode form gives the operator spanning in
Proposition~\ref{prop:empirical-rank}. The proof first takes finite
triangular combinations of rows and columns, after which the
coefficients no longer depend on $\wt u$ or $\wt v$
(Lemma~\ref{lem:formal-weight-removal}). The resulting
coefficients are polynomial evaluations on two triangular grids,
which give both the dimension bound and invertible seed matrices.
Formal independence need not persist in a given $V$
(Remark~\ref{rem:affine-level-one-rank-drop}, $L_1(\mathfrak{sl}_2)$).

\subsection*{Equal inputs, diagonal elimination, and examples}

For $u=v$, the interchange map $J$ is transposition of the two scalar
coefficient indices, so $J^2=1$. Integration by parts identifies its
action on coefficient vectors with an explicit triangular formula. The corresponding
actual-mode interchange holds modulo $U$ and reduces every even
descendant order to smaller orders. The odd-indexed members of
$\mathrm{(A)}$ and $\mathrm{(B)}$ therefore suffice, giving
\begin{equation}\label{eq:intro-equal-input-results}
\dim\mathcal Q_n((u,u))
\le2\left\lfloor\frac{(n+1)^2}{4}\right\rfloor.
\end{equation}
Propositions~\ref{prop:equal-input-odd-spanning}
and~\ref{prop:equal-input-formal-basis} give the operator bound and
the matching formal coefficient dimension, respectively.

For a finite homogeneous strong generating set $S$ with $\ell$
elements, let $\mathcal Q_n(2,S)$ be the span of the maps
$z\mapsto wR^n_{m,k}z+O_n^\circ(2,S,V)^+$ from $V$ to
$V/O_n^\circ(2,S,V)^+$, with $w\in\mathcal P_{\le2}(S)$ and $m\ge k\ge0$.
These bounds combine to give
\begin{equation}\label{eq:intro-global-quadratic-bound}
\dim\mathcal Q_n(2,S)\le
\binom{\ell}{2}n(n+1)
+2\ell\left\lfloor\frac{(n+1)^2}{4}\right\rfloor
\end{equation}
(Corollary~\ref{cor:global-quadratic-bound}). One cross-pair contribution
is needed for each unordered pair of distinct generators.

The coefficient bases also give invertible matrices for two
complementary diagonal eliminations. Lemmas
\ref{lem:generaln-stripA-deep-elim-explicit}
and~\ref{lem:generaln-stripC-deep-elim-explicit} solve for selected deep
quadratic modes with explicit coefficients. At levels one and two,
Heisenberg and universal Virasoro examples attain the equal-input
bounds $2$ and $4$. The affine examples show how additional VOA
relations can reduce the actual rank.

\subsection*{Organization}

Section~\ref{sec:tools-from-part1} recalls the notation and results
of Part~I and proves the numerator identities used below.
Section~\ref{sec:product-commutator} splits the residue of a word into
its product and correction pieces, normalizes the diagonal kernel,
rewrites each correction piece through a contracted word, and proves
the termwise correction theorem; the vanishing of the first correction
is its first step. Section~\ref{sec:quadratic-reduction} begins with the
specialization to a pair (Remark~\ref{rem:bridge-true-regular-quadratic}). It then establishes the
coefficient bases, operator spanning, equal-input and global bounds,
and diagonal eliminations. Section~\ref{sec:examples} explains how to apply
these results to a given strong generating set and carries this out for the
Heisenberg, universal Virasoro, and affine $\mathfrak{sl}_2$ vertex operator
algebras. Section~\ref{sec:conclusion} contains brief
concluding remarks.

\subsection*{Terminology and notation}
The members of the families \textup{(A)}, \textup{(B)}, \textup{(C)} above, and their formal
coefficient vectors \(\mathcal A\), \(\mathcal B\), \(\mathcal C\), are called \emph{seeds};
``strip~A'' and ``strip~C'' refer to the families \textup{(A)} and \textup{(C)} and to the
diagonal eliminations built from them.  The recurring symbols are listed below.
\begin{center}
\renewcommand{\arraystretch}{1.15}
\raggedright
\begin{tabular}{>{\raggedright\arraybackslash}p{.31\textwidth}>{\raggedright\arraybackslash}p{.35\textwidth}>{\raggedright\arraybackslash}p{.25\textwidth}}
Symbol & Meaning & Defined in\\ \hline
$u\,R^n_{m,k}\,v$ & residue operator & \eqref{eq:Rmk-def}\\
$O_n^\circ(u,V)$ & relations from one vector & Def.~\ref{def:Onu-subspace}\\
$U$ & subspace for a pair & \eqref{eq:intro-U-def}, \eqref{eq:U-def}\\
$\mathcal Q_n((u,v))$ & ordered-pair operator space & \eqref{eq:intro-Q-def}\\
$\mathcal Q_n(2,S)$ & global quadratic operator space & \eqref{eq:global-plus-induced-space}\\
$p_{j,r,u}(x)$, $P_{j,r,u}(x)$ & $p$-series and their tails & \eqref{eq:pjmu-def}, \eqref{eq:pm-Pm-decomp}\\
$Y_L(u_{-r}\mathbf1,x)$ in products & derivative field & \eqref{eq:descendant-derivative-convention}\\
$F^+_{i,r,u}(x,x)$ & diagonal kernel & \eqref{eq:Fium-plus-diag}\\
$\Phi_{j,r,u}(x)$ & numerator of $p_{j,r,u}(x)$ & \eqref{eq:Phi-jru}\\
$G^{(n)}_{i,r,u}(x)$ & normalized diagonal kernel & \eqref{eq:G-definition}\\
$O_n^\circ(N,S,V)$, $O_n^\circ(N,S,V)^+$ & finite-length subspaces & Def.~\ref{def:OnN-SV}\\
$U_w$ & subspace for a word's one-sided expansion & \eqref{eq:word-split-modulus}\\
$(\cdot)_{\mathrm p}$, $(\cdot)_{\mathrm m_t,i}$ & product piece and correction pieces & Def.~\ref{def:word-pieces}\\
$y$ & contracted word & \eqref{eq:contracted-word}\\
$C^{(n)}_{m,k}(r;a,b)$ & product coefficients & \eqref{eq:Gamma-discrete2}\\
\textup{(A)}, \textup{(B)}, \textup{(C)} & operator families & \eqref{eq:A-strip}, \eqref{eq:B-strip}, \eqref{eq:C-strip}\\
$\Gamma^{(n)}_{m,k,r}$, $\mathcal S_{\mathrm{prod,reg}}(n;u,v)$ & formal coefficient vectors and span & Def.~\ref{def:universal-prod-reg-space}\\
$\mathcal A$, $\mathcal B$, $\mathcal C$ & formal seed families & \eqref{eq:formal-ABC-seeds}\\
$J$, $\mathcal S_{\mathrm{eq}}(n;u)$ & interchange map, equal-input space & \eqref{eq:equal-input-J-def}, Def.~\ref{def:equal-input-coefficient-space}\\
$U_A$, $U_C$ & subspaces for the eliminations & \eqref{eq:UA-diagonal-elimination}, \eqref{eq:UC-diagonal-elimination}\\
$K_A^{(n)}$, $K_C^{(n)}$, $E^{(n)}_{p,q}$, $F^{(n)}_{\eta,\alpha}$ & elimination matrices and coefficients & \eqref{eq:KA-direct-C-entry}--\eqref{eq:Fka-full-inverse-formula}\\
\end{tabular}
\end{center}

\section{Notation and results from Part~I}\label{sec:tools-from-part1}

This section fixes notation and collects what is used from Part~I \cite{Part1}.
We recall the definitions and formulas needed repeatedly and record one
short lemma.
Table~\ref{tab:part1-results} is a reference guide to their locations in Part~I.
The new material is the set of $\Phi$ identities in Section~\ref{subsec:Phi-basic}.

\subsection{Conventions}
Throughout, $\N=\{0,1,2,\dots\}$.  For $\alpha\in\C$ and $i\in\N$ we use the generalized
binomial coefficients
\begin{equation}\label{eq:how-to-expand}
\binom{\alpha}{i}
=
\frac{\alpha(\alpha-1)\cdots(\alpha-i+1)}{i!},
\end{equation}
with $\binom{\alpha}{0}=1$ and $\binom{\alpha}{k}=0$ for $k<0$, and we expand
\begin{equation}\label{eq:binomial-expansion}
(x+y)^\alpha=\sum_{i\ge 0}\binom{\alpha}{i}x^{\alpha-i}y^i
\end{equation}
in nonnegative powers of $y$.  We write $[x^M]F(x)=\Res_x\,x^{-M-1}F(x)$ for
$F(x)\in\C((x))$ and $M\in\Z$.  These are the conventions of \cite[(2.1)--(2.2)]{Part1}.

The level $n$ Zhu algebra is $A_n(V)=V/O_n(V)$ \cite{Z2,DLM}, where $O_n(V)$ is
spanned by the circle products
$u\circ_nv=\Res_x(1+x)^{\wt u+n}x^{-2n-2}Y(u,x)v$ ($u$ homogeneous) and the
vectors $(L(-1)+L(0))v$ \cite[(2.1)]{DLM}. In the notation of
\cite[(2.2)--(2.3)]{AB}, $O_n(V)=O^L(V)+O_n^\circ(V)$, where
$O^L(V)=\{(L(-1)+L(0))v:\ v\in V\}$ and $O_n^\circ(V)$ is the span of the
circle products.

\begin{defn}[Congruence modulo a subspace]\label{def:W-congruence}
	Let $V$ be a vector space and let $W\subset V$ be a subspace.

	(1) For $a,b\in V$ we write $a\equiv_W b$ if $a-b\in W$.

	(2) For $A,B\in V((x_1))\cdots((x_r))$ we write $A\equiv_W B$ if every coefficient of
	$A-B$ lies in $W$, i.e. $A-B\in W((x_1))\cdots((x_r))$.

	(3) Let $F,G$ be linear maps from $V$ to $V((x_1))\cdots((x_r))$.
	We write $F\equiv_W G$ if
	$(F-G)v\in W((x_1))\cdots((x_r))$ for all $v\in V$; equivalently
	$F-G\in\Hom\!\bigl(V,W((x_1))\cdots((x_r))\bigr)$.
	The Laurent lower bounds may depend on the input vector.
	In particular, for ordinary operators $A,B\in\End(V)$,
	\[
	A\equiv_W B
	\quad\Longleftrightarrow\quad
	(A-B)(V)\subseteq W.
	\]
	Equivalently, $A$ and $B$ induce the same linear map from $V$ to $V/W$.
	These congruences are preserved under coefficient extraction and formal residues
	in the variables $x_1,\dots,x_r$, and under multiplication by scalars in $\C((x_1))\cdots((x_r))$.
\end{defn}

\begin{defn}[Residue operators]\label{def:Rmk}\label{def:Onu-subspace}\label{O-u subspace defn}
	Let $u\in V$ be homogeneous and $n\in\N$.  For $m\ge k\ge0$ and $v\in V$ set
	\begin{equation}\label{eq:Rmk-def}
	u\,R^{n}_{m,k}\,v
	\;:=\;
	\Res_x \frac{(1+x)^{\wt u+n+k}}{x^{m+2n+2}}\,Y(u,x)v,
	\end{equation}
	and let $O_n^\circ(u,V)$ be the span of all $u\,R^n_{m,k}\,v$ with $v\in V$ and $m\ge k\ge0$.
\end{defn}
By \eqref{eq:Rmk-def}, $u\,R^n_{0,0}\,v=u\circ_nv$.  By \cite[Lemma~2.3]{Part1},
$O_n^\circ(u,V)\subseteq O_n^\circ(V)$ and
\begin{equation}\label{eq:Rmk-binomial-reduction}
u\,R^{n}_{m,k}\,v
=
\sum_{t=0}^{k}\binom{k}{t}\,u\,R^{n}_{m-t,0}\,v
\qquad(m\ge k\ge0).
\end{equation}

\subsection{The \texorpdfstring{$p$}{p}-series, \texorpdfstring{$Y_L$}{YL}, and the one-sided expansion}
For homogeneous $u$, $j\ge-2n-1$, and $r\ge1$, Part~I defines
\cite[(2.25), (2.30)]{Part1} the $p$-series
\begin{equation}\label{eq:pjmu-def}
p_{j,r,u}(x):=\frac{1}{(r-1)!}\frac{d^{\,r-1}}{dx^{\,r-1}}
\Biggl(\frac{1}{(1+x)^{\wt u + n}}\sum_{i=0}^{2n+j+1}\binom{\wt u + n}{i}\,x^{\,i-j-1}\Biggr),
\end{equation}
and \cite[Lemma~2.13]{Part1} gives a unique $P_{j,r,u}(x)\in\C[[x]]$ with
\begin{equation}\label{eq:pm-Pm-decomp}
p_{j,r,u}(x)=\binom{-j-1}{r-1}x^{-j-r}+x^{2n+2-r}P_{j,r,u}(x).
\end{equation}
The left truncated field \cite[Definition~2.14]{Part1} is
\begin{equation}\label{eq:YL-pform-def}
Y_L(u,x):=\sum_{j\ge -2n-1} u_j\,p_{j,1,u}(x)\ \in\ \Hom\!\bigl(V,V((x))\bigr),
\end{equation}
extended linearly to all of $V$. In a product of fields, for homogeneous
$u$ and $r\ge1$, the symbol $Y_L(u_{-r}\mathbf1,x)$ denotes the derivative
field \cite[Convention~3.6]{Part1}
\begin{equation}\label{eq:descendant-derivative-convention}
Y_L(u_{-r}\mathbf1,x):=\frac{1}{(r-1)!}\,\partial_x^{\,r-1}Y_L(u,x)
=\sum_{j\ge-2n-1}u_j\,p_{j,r,u}(x),
\end{equation}
the second equality by \eqref{eq:pjmu-def}; for $r=1$ it is $Y_L(u,x)$.
This is a choice of notation: \eqref{eq:descendant-derivative-convention} is
not the left truncated field \eqref{eq:YL-pform-def} of the vector
$u_{-r}\mathbf1$, but as the leftmost factor the two are congruent modulo
$O_n^\circ(u,V)$ by \cite[Lemma~3.5]{Part1}.

For homogeneous $u$, $i\ge0$, and $r\ge1$, the diagonal regular commutator kernel
$F^+_{i,r,u}(x,x)$ of \cite[(3.12)]{Part1} has the explicit form
\cite[Lemma~3.9]{Part1}
\begin{equation}\label{eq:Fium-plus-diag}
F_{i,r,u}^{+}(x,x)
=
\sum_{k=-2n-1}^{-1}\binom{k}{i}\,p_{k,r,u}(x)\,x^{\,k-i}
\;+\;
x^{2n+2-r-i}\sum_{k\ge0}\binom{k}{i}\,P_{k,r,u}(x)\,x^{k},
\end{equation}
with $p_{j,r,u}(x)$ and $P_{j,r,u}(x)$ as in \eqref{eq:pjmu-def} and
\eqref{eq:pm-Pm-decomp}.
For homogeneous $u,v\in V$ and $r\ge1$, set
\[
U_L:=O_n^\circ(u,V)+O_n^\circ(v,V)+\sum_{i\ge0}O_n^\circ(u_iv,V).
\]
The one-step identity \cite[Lemma~3.10]{Part1} states that, in
$\Hom(V,V((x)))$,
\begin{equation}\label{eq:one-step-YL-U}
Y(u_{-r}v,x)
\;\equiv_{\,U_L\,}\;
\sum_{i\ge0} F^+_{i,r,u}(x,x)\,Y_L(u_i v,x)\;+\;Y_L(v,x)\,Y_L(u_{-r}\mathbf1,x),
\end{equation}
where $Y_L(u_{-r}\mathbf1,x)$ is the derivative field of
\eqref{eq:descendant-derivative-convention}.

Iterating the one-step identity along a word gives the one-sided expansion of
Part~I.  Let $N\ge2$, let $s^{(1)},\ldots,s^{(N)}\in V$ be homogeneous, and let
$r_1,\ldots,r_N\ge1$.  Put
\begin{equation}\label{eq:correction-word}
w^{(t)}=s^{(t)}_{-r_t}\cdots s^{(N)}_{-r_N}\mathbf1,
\qquad w=w^{(1)},\qquad
\wt w=\sum_{t=1}^{N}(\wt s^{(t)}+r_t-1),
\end{equation}
and let
\begin{equation}\label{eq:word-split-modulus}
U_w:=\sum_{q=1}^{N}O_n^\circ(s^{(q)},V)
+\sum_{t=0}^{N-2}\sum_{i\ge0}
O_n^\circ(s^{(t+1)}_iw^{(t+2)},V).
\end{equation}
Then \cite[Proposition~3.11]{Part1} states that, in $\Hom(V,V((x)))$,
\begin{equation}\label{eq:one-sided-iterate}
\begin{aligned}
Y(w,x)\equiv_{U_w}{}&
\sum_{t=0}^{N-2}\sum_{i\ge0}
F^+_{i,r_{t+1},s^{(t+1)}}(x,x)\,Y_L(s^{(t+1)}_iw^{(t+2)},x)\,
Y_L(s^{(t)}_{-r_{t}}\mathbf1,x)\cdots
Y_L(s^{(1)}_{-r_1}\mathbf1,x)\\
&+Y_L(s^{(N)}_{-r_N}\mathbf1,x)\cdots
Y_L(s^{(1)}_{-r_1}\mathbf1,x).
\end{aligned}
\end{equation}
Here $t$ is the number of trailing derivative fields, and the
singular field of the summand with index $t$ sits at position $t+1$ of
$w$; this is \cite[(3.23)]{Part1} with its summation index $s=t+1$.
For $t=0$ the trailing product of fields is empty, and each
$Y_L(s^{(q)}_{-r_q}\mathbf1,x)$ is the derivative field of
\eqref{eq:descendant-derivative-convention}.  Each replacement in the proof of
\eqref{eq:one-sided-iterate} is made at a leftmost field.  For $N=2$ and
$r_2=1$, with $s^{(1)}=u$ and $s^{(2)}=v$, the expansion
\eqref{eq:one-sided-iterate} is the one-step identity
\eqref{eq:one-step-YL-U}, and $U_w=U_L$.

\subsection{Replacement residues and finite-length subspaces}
Let $U\subseteq O_n^\circ(V)$ be a subspace, let $v\in V$ be homogeneous, and let
$\widehat Y(v,x)\in\Hom(V,V((x)))$ satisfy $Y(v,x)\equiv_U\widehat Y(v,x)$.  The replacement
residues of \cite[Definition~4.2]{Part1} are
\begin{equation}\label{eq:Rmk-replacement-def}
v\,R^{n,\widehat Y}_{m,k}\,z
\;:=\;
\Res_x \frac{(1+x)^{\wt v+n+k}}{x^{m+2n+2}}\;\widehat Y(v,x)z
\qquad(m\ge k\ge0,\ z\in V),
\end{equation}
and \cite[Proposition~4.3]{Part1} gives
$v\,R^{n,\widehat Y}_{m,k}\,z\equiv_U v\,R^n_{m,k}\,z$, so these residues lie in $O_n^\circ(V)$.

A subset $S\subseteq V$ \emph{strongly generates} $V$ if $V$ is spanned by the vectors
$s^{(1)}_{-m_1}\cdots s^{(r)}_{-m_r}\mathbf1$ with $r\ge0$, $s^{(j)}\in S$, and $m_j\ge1$
\cite[Definition~4.1]{Part1}.

\begin{defn}[Finite-length circle-relation subspaces]\label{def:OnN-SV}
Fix $n\in\N$ and a finite homogeneous subset $S\subset V$.
For $N\ge1$, let $\mathcal P_{\le N}(S)$ consist of the vacuum words
\begin{equation}\label{eq:finite-length-vacuum-words}
w=s^{(1)}_{-m_1}\cdots s^{(r)}_{-m_r}\mathbf1,
\qquad 1\le r\le N,\quad s^{(j)}\in S,\quad
m_1,\ldots,m_{r-1}\in\Z,\quad m_r\ge1.
\end{equation}
Define
\begin{equation}\label{eq:OnNSV-def}
O_n^\circ(N,S,V)
:=\Span_\C\{wR^n_{m,k}z:\ w\in\mathcal P_{\le N}(S),\ z\in V,
\ m\ge k\ge0\}\subseteq O_n^\circ(V).
\end{equation}
The containment follows from \cite[Lemma~2.3]{Part1}.
The plus subspace is
\begin{equation}\label{eq:OnNSV-plus-def}
O_n^\circ(N,S,V)^+
:=\Span_\C\{wR^n_{m,k}z:\ z\in V,\ m\ge k\ge0,
\ w\text{ as specified below}\},
\end{equation}
where $w$ either has length at most $N-1$, or has length $N$ in
\eqref{eq:finite-length-vacuum-words} with $m_j\le0$ for at least one
$j<N$. Thus
\begin{equation}\label{eq:circle-length-inclusions}
O_n^\circ(N-1,S,V)\subseteq O_n^\circ(N,S,V)^+
\subseteq O_n^\circ(N,S,V)\qquad(N\ge2).
\end{equation}
If $S$ strongly generates $V$, then
\begin{equation}\label{eq:circle-length-exhaustion}
\bigcup_{N\ge1}O_n^\circ(N,S,V)=O_n^\circ(V),
\end{equation}
because every vector of $V$ is a finite linear combination of regular
vacuum words and the vacuum slice is zero.
\end{defn}

\begin{lem}[Vacuum descendants do not enlarge $O_n^\circ(\,\cdot\,,V)$]
	\label{lem:descendant-circle-subspace-inclusion}
	Fix $n\in\N$ and let $u\in V$ be homogeneous.  Fix $r\ge 1$.
	Then
	\begin{equation}\label{eq:descendant-On-inclusion}
	O_n^\circ(u_{-r}\mathbf 1,V)\ \subseteq\ O_n^\circ(u,V),
	\end{equation}
	where $O_n^\circ(\,\cdot\,,V)$ is as in Definition~\ref{O-u subspace defn}.
\end{lem}
\begin{proof}
	Since \(u_{-r}\mathbf1=\frac{1}{(r-1)!}L(-1)^{r-1}u\), the claim follows by
	applying \cite[Lemma~2.8]{Part1} \(r-1\) times.
\end{proof}

\subsection{Results used from Part~I}
Table~\ref{tab:part1-results} lists the results of Part~I used in this paper.

\begin{table}[H]
\centering
\renewcommand{\arraystretch}{1.2}
\begin{tabular}{p{0.72\textwidth}l}
Statement & Part~I\\ \hline
$u\,R^n_{m,k}\,v\in O_n^\circ(V)$, and the reduction \eqref{eq:Rmk-binomial-reduction} to $k=0$ & Lemma~2.3\\
$O_n^\circ(L(-1)u,V)\subseteq O_n^\circ(u,V)$ & Lemma~2.8\\
Higher derivatives of the $p$-series and the tails $P_{j,r,u}$, \eqref{eq:pm-Pm-decomp} & Lemma~2.13\\
The left truncated field $Y_L$, \eqref{eq:YL-pform-def} & Definition~2.14\\
$Y_L^-(u,x)=Y^-(u,x)$ and $Y(u,x)\equiv_{O_n^\circ(u,V)}Y_L(u,x)$ & Proposition~3.1\\
Leftmost replacement modulo $O_n^\circ(u_1,V)$ & Lemma~3.3\\
The truncated field of $u_{-r}\mathbf1$ versus the derivative field & Lemma~3.5\\
The derivative-field convention & Convention~3.6\\
Explicit diagonal kernel \eqref{eq:Fium-plus-diag} & Lemma~3.9\\
One-step identity \eqref{eq:one-step-YL-U} & Lemma~3.10\\
One-sided iterate expansion \eqref{eq:one-sided-iterate} and the subspace $U_w$ & Proposition~3.11\\
Strong generating sets & Definition~4.1\\
Replacement residues \eqref{eq:Rmk-replacement-def} and their membership in $O_n^\circ(V)$ & Def.~4.2, Prop.~4.3\\
Finite-length subspaces $O_n^\circ(N,S,V)$ and $O_n^\circ(N,S,V)^+$ & Definition~4.5\\
Termwise correction assertion (proved here) & Conjecture~4.6\\
Expected first-correction inclusion (proved here, Corollary~\ref{cor:Y-comm-residue-in-sumOnuiv}\textup{(iii)}) & (4.10)\\
Expansion of an $N$-iterate vacuum word against $R^n_{m,k}$ & Lemma~4.12\\
First reduction of $O_n^\circ(N,S,V)$: $k=0$ and last mode $-1$ & Corollary~4.14\\
$O_n^\circ(2,\{u,v\},V)^+=O_n^\circ(u,V)+O_n^\circ(v,V)+\sum_{(s,t)}\sum_{i\ge0}O_n^\circ(s_it,V)$, $(s,t)\in\{(u,u),(u,v),(v,v)\}$ & Lemma~4.16\\
The Virasoro quadratic window & Example~5.2\\
\end{tabular}
\caption{Reference guide to the definitions and results of Part~I used in this paper.}
\label{tab:part1-results}
\end{table}

The first goal is to prove the complete termwise correction assertion
of Part~I. The second is to give uniform quadratic spanning families
and combine the pair calculations modulo $O_n^\circ(2,S,V)^+$.
The coefficient bases also supply the two diagonal-elimination formulas.

\subsection{Basic identities for \texorpdfstring{$\Phi$}{Phi}}\label{subsec:Phi-basic}

\begin{defn}[Numerator for $p_{j,r,u}(x)$]\label{def:Phi-jru}
	Fix $n\in\N$ and let $u\in V$ be homogeneous. For integers $r\ge 1$ and
	$j\ge -2n-1$, define
	\begin{equation}\label{eq:Phi-jru}
	\Phi_{j,r,u}(x)
	:=
	\sum_{s=0}^{r-1}
	\binom{-\wt u-n}{\,r-1-s\,}(1+x)^s
	\sum_{i=0}^{2n+j+1}
	\binom{\wt u+n}{\,i\,}\binom{i-j-1}{\,s\,}\,x^{\,i-j-1-s},
	\end{equation}
	so that the Leibniz rule applied to \eqref{eq:pjmu-def} gives
	\begin{equation}\label{eq:p-from-Phi-r}
	p_{j,r,u}(x)
	=
	\frac{1}{(1+x)^{\wt u+n+r-1}}\;\Phi_{j,r,u}(x).
	\end{equation}
\end{defn}
\begin{lem}[Generating-function packaging for $\Phi_{b,r,u}(x)$]
	\label{lem:Phi-y-packaging}
	Fix $n\in\N$ and let $u\in V$ be homogeneous.  Let $r\ge 1$ and $b\ge -2n-1$.
	Then $\Phi_{b,r,u}(x)$ from Definition~\ref{def:Phi-jru} admits the coefficient-extraction form
	\begin{equation}\label{eq:Phi-y-packaging}
	\Phi_{b,r,u}(x)
	=
	\bigl[y^{r-1}\bigr]\,
	(1+y)^{-\wt u-n}\;
	\sum_{p=0}^{2n+b+1}
	\binom{\wt u+n}{p}\,
	\Bigl(x+y(1+x)\Bigr)^{\,p-b-1}.
	\end{equation}
\end{lem}

\begin{proof}
Expand in nonnegative powers of $y$ by \eqref{eq:binomial-expansion}:
\[
\begin{aligned}
&[y^{r-1}](1+y)^{-\wt u-n}\sum_{p=0}^{2n+b+1}
\binom{\wt u+n}p\bigl(x+y(1+x)\bigr)^{p-b-1}\\
&\quad=\sum_{s=0}^{r-1}\binom{-\wt u-n}{r-1-s}(1+x)^s
\sum_{p=0}^{2n+b+1}\binom{\wt u+n}p\binom{p-b-1}s x^{p-b-1-s}\\
&\quad=\Phi_{b,r,u}(x)
\end{aligned}
\]
by Definition~\ref{def:Phi-jru}. Both coefficient convolutions are finite.
\end{proof}

\begin{lem}[Coefficient formula for $\Phi_{j,r,u}(x)$ and the cutoff $2n-d$]
	\label{lem:Phi-coeff-cutoff}
	Fix $n\in\N$ and let $u\in V$ be homogeneous.  Fix $r\ge 1$ and $j\ge -2n-1$.
	Then for every $d\in\Z$ we have
	\begin{equation}\label{eq:Phi-coeff-formula-cutoff}
	[x^d]\Phi_{j,r,u}(x)
	=
	\sum_{t=0}^{\min(r-1,\,2n-d)}
	\binom{\wt u+n}{\,d+j+1+t\,}\,
	\binom{d+t}{t}\,
	\binom{d-(\wt u+n)}{\,r-1-t\,},
	\end{equation}
	the sum being empty when $2n-d<0$.
\end{lem}

\begin{proof}
	In Lemma~\ref{lem:Phi-y-packaging} the variable $y$ only records the
	derivative order, so $[x^d]$ may be taken first. Write
	$x+y(1+x)=x(1+y)+y$ and expand in powers of $y$ by
	\eqref{eq:binomial-expansion}. Only the term $y^t$ with $t=p-j-1-d$
	contributes to $[x^d]$:
	\[
	[x^d]\bigl(x(1+y)+y\bigr)^{p-j-1}=\binom{d+t}{t}(1+y)^d\,y^t,
	\qquad t=p-j-1-d,
	\]
	and this is zero if $t<0$. The range $0\le p\le2n+j+1$ becomes
	$0\le t\le2n-d$, so \eqref{eq:Phi-y-packaging} with $b=j$ gives
	\[
	[x^d]\Phi_{j,r,u}(x)
	=\sum_{t=0}^{2n-d}
	\binom{\wt u+n}{d+j+1+t}\binom{d+t}{t}\,
	[y^{r-1-t}](1+y)^{d-(\wt u+n)}.
	\]
	The last factor is $\binom{d-(\wt u+n)}{r-1-t}$, which vanishes for
	$t>r-1$.
\end{proof}

\begin{cor}[Laurent support and shifted polynomial of $\Phi_{j,r,u}(x)$]
	\label{cor:Phi-Laurent-and-shifted-poly}
	Fix $n\in\N$, let $u\in V$ be homogeneous, and let $r\ge 1$ and $j\ge -2n-1$.
	Then $\Phi_{j,r,u}(x)\in\C[x,x^{-1}]$ is a finite Laurent polynomial supported in degrees
	$-j-r\le d\le 2n$, i.e.
	\begin{equation}\label{eq:Phi-Laurent-expansion}
	\Phi_{j,r,u}(x)=\sum_{d=-j-r}^{2n}	\sum_{t=0}^{\,2n-d\,}
	\binom{\wt u+n}{\,d+j+1+t\,}\,
	\binom{d+t}{\,t\,}\,
	\binom{d-(\wt u+n)}{\,r-1-t\,}x^d.
	\end{equation}
	In particular,
	$x^{j+r}\Phi_{j,r,u}(x)\in\C[x]$ and $\deg\!\bigl(x^{j+r}\Phi_{j,r,u}\bigr)\le 2n+j+r$.
	Moreover, if $j\le -1$ then $\Phi_{j,r,u}(x)\in\C[x]$ and $\deg \Phi_{j,r,u}\le 2n$.
\end{cor}

\begin{proof}
	The bound $d\le2n$ is the emptiness of the sum in
	\eqref{eq:Phi-coeff-formula-cutoff} for $d>2n$. For the lower bound,
	every monomial in \eqref{eq:Phi-jru} has exponent
	$i-j-1-s\ge-j-1-(r-1)=-j-r$. If $j\le-1$, then $i-j-1\ge0$, so
	$\binom{i-j-1}{s}=0$ unless $s\le i-j-1$; every surviving exponent
	therefore lies in $[0,2n]$.
\end{proof}

\section{The diagonal kernel and correction residues}
\label{sec:product-commutator}
The main result of this section is the termwise correction theorem at
every word length, Theorem~\ref{thm:correction-residues-plus}. We start
from the one-sided expansion \eqref{eq:one-sided-iterate}. Applying the
residue kernel splits a residue into a product piece and one correction
piece for each position of the word (Lemma~\ref{lem:word-residue-split}).
After a finite normalization of the diagonal kernel
(Lemma~\ref{lem:G-support}), every correction piece is a finite sum of
product pieces of a contracted word with fewer $Y_L$ factors
(Proposition~\ref{prop:middle-operator-convolution}). At the first
position these product pieces vanish, so the first correction vanishes
(Corollary~\ref{prop:Rcomm-strip-vanish-main}); at the later positions
an induction on the number of $Y_L$ factors shows that they are
relations of shorter words or of singular words on the original letters
(Lemma~\ref{lem:contracted-products-tracked}).

\subsection{Residue pieces of a word}

Fix $n\in\N$ and the word $w$ of \eqref{eq:correction-word}, with its
subspace $U_w$ from \eqref{eq:word-split-modulus}. We keep this displayed
presentation when naming pieces: product pieces are attached to the
ordered inputs, even if the displayed word evaluates to zero. In that
case its assigned weight is the sum in \eqref{eq:correction-word}, and
its full residue is zero.

\begin{defn}[Residue pieces]
\label{def:word-pieces}
Let $m\ge k\ge0$ and $z\in V$. The \emph{product piece} of $w$ is
\begin{equation}\label{eq:word-product-piece}
\begin{aligned}
(wR^n_{m,k}z)_{\mathrm p}
:={}&\Res_x\frac{(1+x)^{\wt w+n+k}}{x^{m+2n+2}}\\
&\quad\cdot Y_L(s^{(N)}_{-r_N}\mathbf1,x)\cdots
Y_L(s^{(1)}_{-r_1}\mathbf1,x)z.
\end{aligned}
\end{equation}
For $0\le t\le N-2$ and $i\ge0$, the \emph{correction piece}
$(wR^n_{m,k}z)_{\mathrm m_t,i}$ is
\begin{equation}\label{eq:word-correction-piece}
\begin{aligned}
(wR^n_{m,k}z)_{\mathrm m_t,i}
:={}&\Res_x\frac{(1+x)^{\wt w+n+k}}{x^{m+2n+2}}
F^+_{i,r_{t+1},s^{(t+1)}}(x,x)
Y_L(s^{(t+1)}_iw^{(t+2)},x)\\
&\quad\cdot Y_L(s^{(t)}_{-r_t}\mathbf1,x)\cdots
Y_L(s^{(1)}_{-r_1}\mathbf1,x)z.
\end{aligned}
\end{equation}
The index $t$ is the number of trailing derivative fields. For $t=0$
the trailing product is empty, that is, it is the identity operator, so
\[
(wR^n_{m,k}z)_{\mathrm m_0,i}
=\Res_x\frac{(1+x)^{\wt w+n+k}}{x^{m+2n+2}}
F^+_{i,r_1,s^{(1)}}(x,x)Y_L(s^{(1)}_iw^{(2)},x)z.
\]
Only finitely many $i$ give $s^{(t+1)}_iw^{(t+2)}\ne0$, and we put
$(wR^n_{m,k}z)_{\mathrm m_t}:=\sum_{i\ge0}(wR^n_{m,k}z)_{\mathrm m_t,i}$.
The piece $(wR^n_{m,k}z)_{\mathrm m_0}$ is the \emph{first correction}.
\end{defn}

The correction piece $(wR^n_{m,k}z)_{\mathrm m_t}$ is the residue of the
summand with index $t$ in \eqref{eq:one-sided-iterate}, whose singular
field sits at position $t+1$ of $w$. The same definition is used with
$N=1$ for a single homogeneous input $y$ of descendant order $1$; there
are then no correction pieces, and
\begin{equation}\label{eq:one-input-product-piece}
(yR^n_{M,k}z)_{\mathrm p}
=\Res_x\frac{(1+x)^{\wt y+n+k}}{x^{M+2n+2}}Y_L(y,x)z
=0\qquad(M\ge k\ge0),
\end{equation}
where the vanishing is \cite[Remark~2.20]{Part1}: the series
$(1+x)^{\wt y+n}Y_L(y,x)z$ has no powers of $x$ above $x^{2n}$, and
$(1+x)^k$ raises the top degree to at most $2n+k<M+2n+1$.

\begin{lem}[Residue split]
\label{lem:word-residue-split}
For $m\ge k\ge0$ and $z\in V$, with $U_w$ the subspace
\eqref{eq:word-split-modulus} and the pieces of
Definition~\ref{def:word-pieces},
\begin{equation}\label{eq:word-residue-split}
wR^n_{m,k}z\equiv_{U_w}
\sum_{t=0}^{N-2}(wR^n_{m,k}z)_{\mathrm m_t}
+(wR^n_{m,k}z)_{\mathrm p}.
\end{equation}
\end{lem}
\begin{proof}
Multiply \eqref{eq:one-sided-iterate}, applied to $z$, by the scalar
kernel $(1+x)^{\wt w+n+k}x^{-m-2n-2}$ and take $\Res_x$. The left side
becomes $wR^n_{m,k}z$ by \eqref{eq:Rmk-def}. The summand with index
$t$ of \eqref{eq:one-sided-iterate} becomes
$(wR^n_{m,k}z)_{\mathrm m_t}$ by \eqref{eq:word-correction-piece}, and
the last term becomes $(wR^n_{m,k}z)_{\mathrm p}$ by
\eqref{eq:word-product-piece}. Both operations preserve congruence
modulo $U_w$ (Definition~\ref{def:W-congruence}\textup{(3)}).
\end{proof}

All retained-mode products are finite after acting on a fixed
homogeneous $z$: every mode index is at least $-2n-1$, and the lower
bound on weights in $V$ bounds their sum above.  General $z$ are
finite sums of homogeneous vectors.

\subsection{A finite normalization of the diagonal kernel}

The complete diagonal kernel has a simple support bound after one
normalization.  The nonnegative-mode tails in
\eqref{eq:Fium-plus-diag} are included in this statement.

\begin{lem}[Finite normalization of the diagonal commutator kernel]
\label{lem:G-support}
For $n\in\N$, homogeneous $u\in V$, $r\ge1$, and $i\ge0$, set
\begin{equation}\label{eq:G-definition}
G^{(n)}_{i,r,u}(x):=(1+x)^{r+i}F^+_{i,r,u}(x,x).
\end{equation}
Then $G^{(n)}_{i,r,u}(x)$ is a Laurent polynomial supported in degrees
$[-r-i,0]$.  Consequently $(1+x)^kG^{(n)}_{i,r,u}(x)$ has degrees
only in $[-r-i,k]$ for $k\ge0$.
\end{lem}
\begin{proof}
The idea is that, for $j\ge0$, the tail $x^{2n+2-r}P_{j,r,u}(x)$ is
$p_{j,r,u}(x)$ minus its principal part, and the principal part is what
the full binomial series in \eqref{eq:pjmu-def} would contribute. After
multiplying by $(1+x)^{\wt u+n+r-1}$, the band and the tails therefore
combine into finite signed sums over $j$ whose dependence on the
binomial index is polynomial, and such sums resum to powers of $1+x$.

By \eqref{eq:pm-Pm-decomp}, the second term in
\eqref{eq:Fium-plus-diag} is
\[
\sum_{j\ge0}\binom ji x^{j-i}
\left(p_{j,r,u}(x)-\binom{-j-1}{r-1}x^{-j-r}\right).
\]
Using
\eqref{eq:p-from-Phi-r}, extend the binomial sum in
\eqref{eq:Phi-jru} to its full series.  Its full contribution is
\[
\binom{-j-1}{r-1}x^{-j-r}(1+x)^{\wt u+n+r-1},
\]
because the full $p$-series before differentiating is $x^{-j-1}$.
Thus, for each binomial-series index $\ell$, the retained contribution
has the indicator $j\ge\ell-2n-1$, whereas the subtracted full
contribution has the indicator $j\ge0$.  Combining these indicators
with the finite negative-$j$ range gives
\begin{equation}\label{eq:G-binomial}
\begin{aligned}
x^{i+r}(1+x)^{\wt u+n+r-1}F^+_{i,r,u}(x,x)
={}&\sum_{\ell\ge0}\binom{\wt u+n}{\ell}x^\ell
\sum_{s=0}^{r-1}\binom{-\wt u-n}{r-1-s}\\
&\quad\cdot(1+x)^s x^{r-1-s}
\sum_{j=\ell-2n-1}^{-1}\binom ji\binom{\ell-j-1}{s}.
\end{aligned}
\end{equation}
Here $\sum_{j=A}^{-1}=-\sum_{j=0}^{A-1}$ if $A>0$; the sum is
empty if $A=0$.  This is a coefficientwise identity: for fixed
$\ell$, the difference of the indicators leaves only finitely many
$j$.  Also $\ell+r-1-s\ge\ell$, so each coefficient of the right
side involves only finitely many $\ell$.

For fixed $i,s,n$, the signed endpoint sum in
\eqref{eq:G-binomial} is a polynomial in $\ell$ of degree at most
$i+s+1$.  Indeed, its summand is a polynomial in $j,\ell$ of total
degree at most $i+s$; a polynomial antidifference in $j$ has total
degree at most $i+s+1$, and is evaluated at two affine endpoints.
Write that polynomial as
$\sum_{q=0}^{i+s+1}c_{s,q}\binom\ell q$.  The formal identity
\[
\sum_{\ell\ge0}\binom{\wt u+n}{\ell}\binom\ell q x^\ell
=\binom{\wt u+n}{q}x^q(1+x)^{\wt u+n-q}
\]
holds for every integer value of $\wt u+n$.  Substitution in \eqref{eq:G-binomial} therefore gives
\begin{equation}\label{eq:G-finite}
\begin{aligned}
G^{(n)}_{i,r,u}(x)
={}&\sum_{s=0}^{r-1}\sum_{q=0}^{i+s+1}
\binom{-\wt u-n}{r-1-s}c_{s,q}\binom{\wt u+n}{q}\\
&\qquad\cdot
x^{-(i+1+s-q)}(1+x)^{i+1+s-q}.
\end{aligned}
\end{equation}
Every exponent $i+1+s-q$ lies between $0$ and $i+r$.
Each summand has degrees between $-i-r$ and $0$, proving the
claim.  Multiplication by the polynomial $(1+x)^k$ gives the
second assertion.
\end{proof}

The point of the normalization is that $G^{(n)}_{i,r,u}$ has no
positive powers of $x$.  It need not be an ordinary polynomial.
For example, when $n=1$, $r=1$, $i=0$, and $\wt u=1$,
\eqref{eq:Fium-plus-diag} and \eqref{eq:Phi-jru} give
\[
F^+_{0,1,u}(x,x)=\frac{3+x}{x(1+x)},
\qquad G^{(1)}_{0,1,u}(x)=3x^{-1}+1.
\]

\subsection{Contracted words and the shortening identity}

Compare the correction piece \eqref{eq:word-correction-piece} with a
product piece. Its fields are $Y_L(s^{(t+1)}_iw^{(t+2)},x)$ followed by
the $t$ derivative fields of the first $t$ letters. These are exactly
the fields of the product piece of a word with $t+1$ ordered inputs,
whose last input is the singular vector $s^{(t+1)}_iw^{(t+2)}$. Only the
scalar kernels differ, and Lemma~\ref{lem:G-support} shows that the
quotient of the two kernels is a Laurent polynomial. We first name the
shorter word.

Fix $0\le t\le N-2$ and $i\ge0$. The \emph{contraction of $w$ at
position $t+1$ with singular index $i$} is the displayed word
\begin{equation}\label{eq:contracted-word}
\begin{aligned}
y&:=s^{(1)}_{-r_1}\cdots s^{(t)}_{-r_t}\,v,
\qquad v:=s^{(t+1)}_iw^{(t+2)},\\
\text{that is,}\quad
y&=s^{(1)}_{-r_1}\cdots s^{(t)}_{-r_t}\,
s^{(t+1)}_{i}\,s^{(t+2)}_{-r_{t+2}}\cdots s^{(N)}_{-r_N}\mathbf1 .
\end{aligned}
\end{equation}
It is obtained from $w$ by replacing the mode $-r_{t+1}$ of the letter
$s^{(t+1)}$ by the nonnegative mode $i$. For $t=0$ we have $y=v$. We
regard $y$ as a word with the $t+1$ ordered homogeneous inputs
$s^{(1)},\ldots,s^{(t)},v$ and descendant orders $r_1,\ldots,r_t,1$,
since $v=v_{-1}\mathbf1$. Its product piece is
Definition~\ref{def:word-pieces} for this presentation:
\begin{equation}\label{eq:contracted-product-piece}
\begin{aligned}
(yR^n_{M,0}z)_{\mathrm p}
={}&\Res_x\frac{(1+x)^{\wt y+n}}{x^{M+2n+2}}\\
&\quad\cdot Y_L(v,x)\,Y_L(s^{(t)}_{-r_t}\mathbf1,x)\cdots
Y_L(s^{(1)}_{-r_1}\mathbf1,x)z\qquad(M\ge0),
\end{aligned}
\end{equation}
where $Y_L(v_{-1}\mathbf1,x)=Y_L(v,x)$ by
\eqref{eq:descendant-derivative-convention}. For $t=0$ this is
\eqref{eq:one-input-product-piece} with $y=v$ and $k=0$.

Since $\wt(a_jb)=\wt a+\wt b-j-1$ for homogeneous $a,b$,
\[
\wt\bigl(s^{(t+1)}_{-r_{t+1}}w^{(t+2)}\bigr)
-\wt\bigl(s^{(t+1)}_{i}w^{(t+2)}\bigr)
=(r_{t+1}-1)-(-i-1)=r_{t+1}+i,
\]
and the modes $s^{(1)}_{-r_1},\ldots,s^{(t)}_{-r_t}$ raise both weights by
the same amount. Hence
\begin{equation}\label{eq:contracted-weight}
\wt w-\wt y=r_{t+1}+i.
\end{equation}
With \eqref{eq:G-definition}, this gives the kernel identity
\begin{equation}\label{eq:kernel-normalization}
\begin{aligned}
(1+x)^{\wt w+n+k}F^+_{i,r_{t+1},s^{(t+1)}}(x,x)
&=(1+x)^{\wt y+n}(1+x)^{k}(1+x)^{r_{t+1}+i}F^+_{i,r_{t+1},s^{(t+1)}}(x,x)\\
&=(1+x)^{\wt y+n}(1+x)^{k}G^{(n)}_{i,r_{t+1},s^{(t+1)}}(x).
\end{aligned}
\end{equation}

\begin{prop}[Shortening identity]
\label{prop:middle-operator-convolution}
Let $0\le t\le N-2$, $i\ge0$, $m\ge k\ge0$, and $z\in V$, and let $y$
be the contracted word \eqref{eq:contracted-word}. Then, as an exact
identity in $V$,
\begin{equation}\label{eq:middle-operator-convolution}
(wR^n_{m,k}z)_{\mathrm m_t,i}
=\sum_{\ell=-r_{t+1}-i}^{k}
[x^\ell]\bigl((1+x)^kG^{(n)}_{i,r_{t+1},s^{(t+1)}}(x)\bigr)\,
(yR^n_{m-\ell,0}z)_{\mathrm p}.
\end{equation}
Every shifted index is admissible: $m-\ell\ge m-k\ge0$. If
$s^{(t+1)}_iw^{(t+2)}=0$, both sides are zero.
\end{prop}
\begin{proof}
Substitute \eqref{eq:kernel-normalization} into
\eqref{eq:word-correction-piece}. By Lemma~\ref{lem:G-support},
\[
(1+x)^kG^{(n)}_{i,r_{t+1},s^{(t+1)}}(x)
=\sum_{\ell=-r_{t+1}-i}^{k}c_\ell x^\ell,
\qquad
c_\ell:=[x^\ell]\bigl((1+x)^kG^{(n)}_{i,r_{t+1},s^{(t+1)}}(x)\bigr),
\]
a finite sum. Since $x^{\ell}x^{-m-2n-2}=x^{-(m-\ell)-2n-2}$, the
correction piece equals
\[
\sum_{\ell=-r_{t+1}-i}^{k}c_\ell\,
\Res_x\frac{(1+x)^{\wt y+n}}{x^{(m-\ell)+2n+2}}
Y_L(v,x)Y_L(s^{(t)}_{-r_t}\mathbf1,x)\cdots
Y_L(s^{(1)}_{-r_1}\mathbf1,x)z,
\]
with $v=s^{(t+1)}_iw^{(t+2)}$. By \eqref{eq:contracted-product-piece}
the residue in the summand indexed by $\ell$ is
$(yR^n_{m-\ell,0}z)_{\mathrm p}$.
\end{proof}

The shortening concerns the number of $Y_L$ factors, not the number of
original letters: the contracted input $v$ still contains its original
suffix letters and its singular first mode. The identity uses only the
homogeneity of the ordered inputs, so it applies to any word of the form
\eqref{eq:correction-word} in homogeneous letters, in particular to a
contracted word with the vector $v$ as its last letter.

\begin{cor}[Full fields and the first correction]
\label{prop:Rcomm-strip-vanish-main}
\label{cor:Y-comm-residue-in-sumOnuiv}
\label{cor:first-correction-arbitrary-length}
Let $m\ge k\ge0$, $z\in V$, $0\le t\le N-2$, and $i\ge0$, and put
$v=s^{(t+1)}_iw^{(t+2)}$.
\begin{enumerate}[label=\textup{(\roman*)}]
\item Replacing the leftmost field $Y_L(v,x)$ in
\eqref{eq:word-correction-piece} by the full field $Y(v,x)$ changes the
correction piece only by an element of $O_n^\circ(v,V)$:
\begin{equation}\label{eq:full-field-correction}
\begin{aligned}
&\Res_x\frac{(1+x)^{\wt w+n+k}}{x^{m+2n+2}}
F^+_{i,r_{t+1},s^{(t+1)}}(x,x)Y(v,x)\\
&\qquad\cdot Y_L(s^{(t)}_{-r_t}\mathbf1,x)\cdots
Y_L(s^{(1)}_{-r_1}\mathbf1,x)z
\;\equiv_{O_n^\circ(v,V)}\;(wR^n_{m,k}z)_{\mathrm m_t,i}.
\end{aligned}
\end{equation}
\item The first correction vanishes: for every $i\ge0$,
\[
(wR^n_{m,k}z)_{\mathrm m_0,i}=0\quad\text{in }V,
\qquad\text{hence}\qquad
(wR^n_{m,k}z)_{\mathrm m_0}=0.
\]
\item With the full field,
\[
\Res_x\frac{(1+x)^{\wt w+n+k}}{x^{m+2n+2}}
\sum_{i\ge0}F^+_{i,r_1,s^{(1)}}(x,x)Y(s^{(1)}_iw^{(2)},x)z
\in\sum_{i\ge0}O_n^\circ(s^{(1)}_iw^{(2)},V),
\]
and each fixed-$i$ summand lies in its own slice
$O_n^\circ(s^{(1)}_iw^{(2)},V)$. For $N=2$ this is the expected inclusion
\cite[(4.10)]{Part1}.
\end{enumerate}
\end{cor}
\begin{proof}
(i) By \cite[Lemma~3.7(a)]{Part1}, the vector
$Y_L(s^{(t)}_{-r_t}\mathbf1,x)\cdots Y_L(s^{(1)}_{-r_1}\mathbf1,x)z$
is a finite sum of terms $f(x)z'$ with $f(x)\in\C((x))$ and $z'\in V$;
for $t=0$ it is $z$. For each such $z'$, leftmost replacement
\cite[Lemma~3.3]{Part1} gives
$Y(v,x)z'\equiv_{O_n^\circ(v,V)}Y_L(v,x)z'$. Multiplication by the
scalar series $f(x)$ and by the scalar kernel, and extraction of the
residue, preserve this congruence
(Definition~\ref{def:W-congruence}\textup{(3)}).

(ii) For $t=0$ the contracted word is $y=v$, a single input. By
Proposition~\ref{prop:middle-operator-convolution},
$(wR^n_{m,k}z)_{\mathrm m_0,i}$ is a finite linear combination of the
residues $(vR^n_{m-\ell,0}z)_{\mathrm p}$ with $m-\ell\ge0$, and each is
zero by \eqref{eq:one-input-product-piece}.

(iii) Combine (i) and (ii) with $t=0$, and sum over the finitely many
$i$ with $s^{(1)}_iw^{(2)}\ne0$.
\end{proof}

\subsection{Letter tracking and the termwise theorem}

From now on $S\subset V$ is a finite homogeneous set and the letters
$s^{(1)},\ldots,s^{(N)}$ of $w$ lie in $S$. Every word that occurs below
is a displayed vacuum word in letters of $S$, so whether its relations
lie in $O_n^\circ(N,S,V)^+$ can be read off from
Definition~\ref{def:OnN-SV}. We record slightly more, namely which
letters occur; Part~III \cite{Part3} uses this information.

\begin{defn}[Residues retaining the original letters]
\label{def:letter-tracked-residues}
Relative to the displayed word $w$ in \eqref{eq:correction-word},
call a residue $w'R^n_{m',k'}z'$, with $m'\ge k'\ge0$ and $z'\in V$,
\emph{letter-tracked} if its displayed left-input word $w'$ has
one of the following forms:
\begin{enumerate}[label=\textup{(T\arabic*)}]
\item its letters form a proper sub-multiset of the letters of $w$;
\item it has all the letters of $w$, counted with multiplicity,
and a nonnegative mode before its last letter.
\end{enumerate}
The last mode of $w'$ is negative in both cases, as in
\eqref{eq:finite-length-vacuum-words}.  A vector is letter-tracked
if it is a finite sum of such residues.  Type~\textup{(T1)}
residues lie in $O_n^\circ(N-1,S,V)$ and type~\textup{(T2)}
residues lie in $O_n^\circ(N,S,V)^+$, by
Definition~\ref{def:OnN-SV}; hence every letter-tracked vector belongs
to $O_n^\circ(N,S,V)^+$.
\end{defn}

The contracted words \eqref{eq:contracted-word} have all the letters of
$w$ and the nonnegative mode $i$ at position $t+1\le N-1$, so their
residues have type~\textup{(T2)}. Contracting again produces words of
the same kind, with possibly several nonnegative or arbitrary modes
after the first contracted position. The next lemma handles all of
them at once.

\begin{lem}[Product pieces of contracted words]
\label{lem:contracted-products-tracked}
Let $1\le L\le N-1$, and let $v$ be a displayed word in the letters
$s^{(L)},\ldots,s^{(N)}$, in this order and with multiplicity, of the
form
\begin{equation}\label{eq:contracted-suffix}
v=s^{(L)}_{j_L}s^{(L+1)}_{j_{L+1}}\cdots s^{(N-1)}_{j_{N-1}}
s^{(N)}_{-r_N}\mathbf1,
\qquad j_L\ge0,\quad j_{L+1},\ldots,j_{N-1}\in\Z.
\end{equation}
Put $y:=s^{(1)}_{-r_1}\cdots s^{(L-1)}_{-r_{L-1}}v$, regarded as a word
with the $L$ ordered inputs $s^{(1)},\ldots,s^{(L-1)},v$ and descendant
orders $r_1,\ldots,r_{L-1},1$. Then for all $m'\ge k'\ge0$ and $z\in V$,
the product piece $(yR^n_{m',k'}z)_{\mathrm p}$ is letter-tracked
relative to $w$ in the sense of Definition~\ref{def:letter-tracked-residues}.
\end{lem}
\begin{proof}
We induct on $L$. If $L=1$, then $y=v$ is a single input and
$(vR^n_{m',k'}z)_{\mathrm p}=0$ by \eqref{eq:one-input-product-piece}.

Let $2\le L\le N-1$ and assume the claim for all smaller values of $L$.
Apply Lemma~\ref{lem:word-residue-split} to $y$, with letters
$s^{(1)},\ldots,s^{(L-1)},v$ and descendant orders
$r_1,\ldots,r_{L-1},1$; these play the roles of $s^{(1)},\ldots,s^{(N)}$
and $r_1,\ldots,r_N$ in \eqref{eq:correction-word}, with $N$ replaced by
$L$. The suffixes of $y$, playing the role of the $w^{(q)}$ in
\eqref{eq:correction-word}, are
\[
y^{(L)}:=v,\qquad
y^{(q)}:=s^{(q)}_{-r_q}\cdots s^{(L-1)}_{-r_{L-1}}v\quad(1\le q\le L-1),
\]
so \eqref{eq:word-residue-split} for $y$ reads
\begin{equation}\label{eq:contracted-split}
(yR^n_{m',k'}z)_{\mathrm p}
=yR^n_{m',k'}z
-\sum_{t=0}^{L-2}(yR^n_{m',k'}z)_{\mathrm m_t}
-e,\qquad
e\in U_y,
\end{equation}
where
\[
U_y:=\sum_{q=1}^{L-1}O_n^\circ(s^{(q)},V)+O_n^\circ(v,V)
+\sum_{t=0}^{L-2}\sum_{j\ge0}O_n^\circ\bigl(s^{(t+1)}_jy^{(t+2)},V\bigr)
\]
is the subspace \eqref{eq:word-split-modulus} for $y$. Its last summand,
for a given $t$, consists of the slices of the singular inputs
$s^{(t+1)}_jy^{(t+2)}$ of the correction pieces
$(yR^n_{m',k'}z)_{\mathrm m_t,j}$ in \eqref{eq:word-correction-piece}. We check that each term on the
right of \eqref{eq:contracted-split} is letter-tracked in the sense of
Definition~\ref{def:letter-tracked-residues}.

\emph{The full residue.} As a displayed vacuum word,
$y=s^{(1)}_{-r_1}\cdots s^{(L-1)}_{-r_{L-1}}s^{(L)}_{j_L}\cdots
s^{(N-1)}_{j_{N-1}}s^{(N)}_{-r_N}\mathbf1$ has exactly the letters of
$w$, the nonnegative mode $j_L$ at position $L\le N-1$, and last mode
$-r_N$. Hence $yR^n_{m',k'}z$ has type~\textup{(T2)}.

\emph{The subspace $U_y$.} The slice $O_n^\circ(s^{(q)},V)$ is spanned by
residues of the one-letter word $s^{(q)}_{-1}\mathbf1$, of
type~\textup{(T1)} since $N\ge2$. The word $v$ has the letters
$s^{(L)},\ldots,s^{(N)}$, a proper sub-multiset because $L\ge2$, so
$O_n^\circ(v,V)$ is of type~\textup{(T1)}. The singular input
$s^{(t+1)}_jy^{(t+2)}$ of the correction piece $\mathrm m_t$ of $y$ has
the letters $s^{(t+1)},\ldots,s^{(N)}$ and last mode $-r_N$. For
$t\ge1$ these form a proper sub-multiset of the letters of $w$, so its
slice has type~\textup{(T1)}. For $t=0$, the input of the first
correction of $y$, the word $s^{(1)}_jy^{(2)}$ has all the letters of
$w$ and the nonnegative first mode $j$ before its last letter, so its
slice has type~\textup{(T2)}.

\emph{The correction pieces.} Fix $0\le t\le L-2$ and $j\ge0$. For
$t=0$ the piece is zero by
Corollary~\ref{prop:Rcomm-strip-vanish-main}\textup{(ii)}, applied to
$y$; the following argument covers all $t$. By
Proposition~\ref{prop:middle-operator-convolution}, applied to $y$,
$(yR^n_{m',k'}z)_{\mathrm m_t,j}$ is a finite linear combination of
product pieces $(y'R^n_{m'-\ell,0}z)_{\mathrm p}$ with $m'-\ell\ge0$,
where
\[
y'=s^{(1)}_{-r_1}\cdots s^{(t)}_{-r_t}\,v',\qquad
v'=s^{(t+1)}_{j}\,s^{(t+2)}_{-r_{t+2}}\cdots
s^{(L-1)}_{-r_{L-1}}\,v
\]
is the contraction of $y$ at position $t+1$ (for $t=L-2$,
$v'=s^{(L-1)}_jv$). The word $v'$ has the form
\eqref{eq:contracted-suffix} with $L$ replaced by $t+1\le L-1$: its
letters are $s^{(t+1)},\ldots,s^{(N)}$ in order, its first mode $j$ is
nonnegative, and its last mode is $-r_N$. By the induction hypothesis
each $(y'R^n_{m'-\ell,0}z)_{\mathrm p}$ is letter-tracked, and summing
over the finitely many $j$ with $s^{(t+1)}_jy^{(t+2)}\ne0$ shows that
$(yR^n_{m',k'}z)_{\mathrm m_t}$ is letter-tracked.

Therefore every term on the right of \eqref{eq:contracted-split} is
letter-tracked, and so is $(yR^n_{m',k'}z)_{\mathrm p}$.
\end{proof}

We now prove the correction-residue conjecture of Part~I,
\cite[Conjecture~4.6]{Part1}, in the following stronger form: the
set $S$ need not strongly generate $V$, and the errors are
letter-tracked.

\begin{thm}[Correction residues in the plus subspace]
\label{thm:correction-residues-plus}
\label{conj:finite-level-revised}\label{conj:correction-residues-plus}
Let $n\in\N$, let $S\subset V$ be a finite homogeneous set, and
let $w$ be the regular word \eqref{eq:correction-word} of length
$N\ge2$ with letters in $S$.  For $m\ge k\ge0$, $z\in V$, and
$0\le t\le N-2$, the full-field correction residue satisfies
\begin{equation}\label{eq:conj-termwise-correction-mod-plus}
\begin{aligned}
&\Res_x\frac{(1+x)^{\wt w+n+k}}{x^{m+2n+2}}
\sum_{i\ge0}F^+_{i,r_{t+1},s^{(t+1)}}(x,x)
Y(s^{(t+1)}_iw^{(t+2)},x)\\
&\qquad\cdot
Y_L(s^{(t)}_{-r_t}\mathbf1,x)\cdots
Y_L(s^{(1)}_{-r_1}\mathbf1,x)z
\in O_n^\circ(N,S,V)^+.
\end{aligned}
\end{equation}
For $t=0$ the trailing product is absent.  The assertion holds
separately for every fixed singular index $i$, and also with
$Y_L$ in place of the leftmost $Y$.

More precisely, each such residue and the difference
\begin{equation}\label{eq:tracked-product-error}
wR^n_{m,k}z-(wR^n_{m,k}z)_{\mathrm p}
\end{equation}
are letter-tracked relative to $w$
(Definition~\ref{def:letter-tracked-residues}).  Consequently,
\begin{equation}\label{eq:word-product-mod-plus}
wR^n_{m,k}z\equiv_{O_n^\circ(N,S,V)^+}
(wR^n_{m,k}z)_{\mathrm p}.
\end{equation}
In particular \cite[Conjecture~4.6]{Part1} holds: it is
\eqref{eq:conj-termwise-correction-mod-plus} under the additional
assumption that $S$ strongly generates $V$.
\end{thm}
\begin{proof}
Fix $t$ and $i$, put $v=s^{(t+1)}_iw^{(t+2)}$, and let $y$ be the
contracted word \eqref{eq:contracted-word}.

\emph{The $Y_L$ correction pieces.} For $t=0$ these are zero by
Corollary~\ref{prop:Rcomm-strip-vanish-main}\textup{(ii)}. For every
$t$, by
Proposition~\ref{prop:middle-operator-convolution},
$(wR^n_{m,k}z)_{\mathrm m_t,i}$ is a finite linear combination of the
product pieces $(yR^n_{m-\ell,0}z)_{\mathrm p}$, $m-\ell\ge0$. The last
input $v$ of $y$ has the form \eqref{eq:contracted-suffix} with
$L=t+1\le N-1$, $j_{t+1}=i$, and $j_q=-r_q$ for $q>t+1$. By
Lemma~\ref{lem:contracted-products-tracked} these product pieces are
letter-tracked.

\emph{The full field.} By
Corollary~\ref{prop:Rcomm-strip-vanish-main}\textup{(i)}, replacing
$Y_L(v,x)$ by $Y(v,x)$ changes the residue by an element of
$O_n^\circ(v,V)$. The word $v$ has the letters
$s^{(t+1)},\ldots,s^{(N)}$ and last mode $-r_N$. For $t\ge1$ this is a
proper sub-multiset of the letters of $w$, so $O_n^\circ(v,V)$ has
type~\textup{(T1)}. For $t=0$ it has all the letters of $w$ and the
nonnegative first mode $i$ before its last letter (since $N\ge2$), so it
has type~\textup{(T2)}. Summing over the finitely many $i$ with $v\ne0$
proves \eqref{eq:conj-termwise-correction-mod-plus}, in the
letter-tracked form.

\emph{The difference \eqref{eq:tracked-product-error}.} The summands of
the subspace $U_w$ in \eqref{eq:word-split-modulus} are letter-tracked:
$O_n^\circ(s^{(q)},V)$ has type~\textup{(T1)}, and the slice of the
singular input $s^{(t+1)}_iw^{(t+2)}$ of $\mathrm m_t$ has
type~\textup{(T1)} for $t\ge1$ and type~\textup{(T2)} for $t=0$, as in
the previous paragraph.
By Lemma~\ref{lem:word-residue-split},
$wR^n_{m,k}z-(wR^n_{m,k}z)_{\mathrm p}$ is the sum of an element of
$U_w$ and the correction pieces $(wR^n_{m,k}z)_{\mathrm m_t}$, all of
which are letter-tracked. Every letter-tracked vector lies in
$O_n^\circ(N,S,V)^+$ (Definition~\ref{def:letter-tracked-residues}), which gives
\eqref{eq:word-product-mod-plus}.
\end{proof}

At level zero the pieces vanish outright.

\begin{lem}[Level zero]
\label{rem:level-zero-pieces}
Let $n=0$, $m\ge k\ge0$, and $z\in V$, and let $w$ be the word
\eqref{eq:correction-word} in arbitrary homogeneous letters. Then
\begin{equation}\label{eq:level-zero-pieces}
(wR^0_{m,k}z)_{\mathrm p}=0,
\qquad
(wR^0_{m,k}z)_{\mathrm m_t,i}=0
\quad(0\le t\le N-2,\ i\ge0).
\end{equation}
Consequently $wR^0_{m,k}z\in U_w$, where $U_w$ is the subspace
\eqref{eq:word-split-modulus}.
\end{lem}
\begin{proof}
At $n=0$ the cutoff is $-2n-1=-1$. By
\eqref{eq:descendant-derivative-convention} and \eqref{eq:p-from-Phi-r},
\[
Y_L(s^{(q)}_{-r_q}\mathbf1,x)
=\sum_{a\ge-1}s^{(q)}_a\,
(1+x)^{-(\wt s^{(q)}+r_q-1)}\Phi_{a,r_q,s^{(q)}}(x).
\]
Substituting into \eqref{eq:word-product-piece} and using
$\wt w=\sum_{q=1}^N(\wt s^{(q)}+r_q-1)$ from \eqref{eq:correction-word},
all powers of $1+x$ except $(1+x)^k$ cancel, and
\[
(wR^0_{m,k}z)_{\mathrm p}
=\sum_{a_1,\ldots,a_N\ge-1}
[x^{m+1}]\Bigl((1+x)^k\prod_{q=1}^N\Phi_{a_q,r_q,s^{(q)}}(x)\Bigr)\,
s^{(N)}_{a_N}\cdots s^{(1)}_{a_1}z,
\]
a finite sum. By Corollary~\ref{cor:Phi-Laurent-and-shifted-poly} with
$n=0$, each $\Phi_{a_q,r_q,s^{(q)}}(x)$ has no powers of $x$ above
$x^{0}$, so the product inside the brackets has no powers above $x^k$.
Since $m+1>k$, every coefficient is zero. This proves the first
equation, for every word in homogeneous letters. By
Proposition~\ref{prop:middle-operator-convolution}, each
$(wR^0_{m,k}z)_{\mathrm m_t,i}$ is a linear combination of product
pieces $(yR^0_{m-\ell,0}z)_{\mathrm p}$ of the contracted word $y$,
which vanish by the first equation applied to $y$. The last statement
follows from \eqref{eq:word-residue-split} in
Lemma~\ref{lem:word-residue-split}, since all pieces on its right side
vanish.
\end{proof}

\begin{cor}[Individual pieces are circle relations]
\label{cor:pieces-circle}
For every $n\in\N$, every individual correction piece
$(wR^n_{m,k}z)_{\mathrm m_t,i}$ and every product piece
$(wR^n_{m,k}z)_{\mathrm p}$ belongs to $O_n^\circ(V)$.
\end{cor}
\begin{proof}
By Proposition~\ref{prop:middle-operator-convolution}, a correction piece is a
linear combination of product pieces of contracted words, so it suffices to
treat product pieces. For a product piece, the full residue
$wR^n_{m,k}z$ and the subspace $U_w$ of \eqref{eq:word-split-modulus} lie in $O_n^\circ(V)$, so
\eqref{eq:word-residue-split} expresses $(wR^n_{m,k}z)_{\mathrm p}$ as an
element of $O_n^\circ(V)$ minus correction pieces. Induct on the number
$N$ of ordered inputs: each correction piece of $w$ is a combination of product
pieces of words with at most $N-1$ inputs, and for $N=1$ the product piece is
zero by \eqref{eq:one-input-product-piece}.
\end{proof}

\section{Quadratic coefficient bases, spanning, and diagonal elimination}
\label{sec:quadratic-reduction}

\begin{rem}[The quadratic pair]\label{rem:bridge-true-regular-quadratic}
	\label{subsec:quadratic-split}
	\label{lem:prod-comm-split}\label{def:prod-comm-split-mk}
	For $N=2$ and $r_2=1$ write $u=s^{(1)}$, $v=s^{(2)}$, and $r=r_1$, so
	$w=u_{-r}v$ and $\wt w=\wt u+\wt v+r-1$; here $v$ may be any homogeneous
	vector. The subspace \eqref{eq:word-split-modulus} is then the
	subspace attached to the pair
	\begin{equation}\label{eq:U-def}
	U\;:=\;O_n^\circ(u,V)+O_n^\circ(v,V)+\sum_{i\ge0}O_n^\circ(u_i v,V)
	\ \subseteq\ O_n^\circ(V),
	\end{equation}
	and the pieces of Definition~\ref{def:word-pieces} are the
	\emph{product contribution} and the first correction, which for a pair we
	also call the \emph{commutator contribution}:
	\begin{align}
	\Bigl((u_{-r}v)R^n_{m,k}z\Bigr)_{\mathrm{p}}
	={}&\Res_x\frac{(1+x)^{\wt u+\wt v+r-1+n+k}}{x^{m+2n+2}}\,
	Y_L(v,x)\,Y_L(u_{-r}\mathbf 1,x)\,z,
	\label{eq:prod-piece-def-mk}\\[4pt]
	\Bigl((u_{-r}v)R^n_{m,k}z\Bigr)_{\mathrm{m}_0}
	={}&\Res_x\frac{(1+x)^{\wt u+\wt v+r-1+n+k}}{x^{m+2n+2}}
	\sum_{i\ge0} F_{i,r,u}^{+}(x,x)\,Y_L(u_i v,x)\,z.
	\label{eq:comm-piece-def-mk}
	\end{align}
	For $N=2$ the only correction piece is the first one ($t=0$). Lemma~\ref{lem:word-residue-split}
	and Corollary~\ref{prop:Rcomm-strip-vanish-main}\textup{(ii)} give, for all
	homogeneous $u,v$, $z\in V$, $r\ge1$, and $m\ge k\ge0$,
	\begin{equation}\label{eq:prod-comm-split-congruence-mk}
	(u_{-r}v)\,R^n_{m,k}z
	\;\equiv_{\,U\,}\;
	\Bigl((u_{-r}v)R^n_{m,k}z\Bigr)_{\mathrm{p}}
	+\Bigl((u_{-r}v)R^n_{m,k}z\Bigr)_{\mathrm{m}_0},
	\qquad
	\Bigl((u_{-r}v)R^n_{m,k}z\Bigr)_{\mathrm{m}_0}=0,
	\end{equation}
	and therefore
	\begin{equation}\label{eq:drop-comm-strip}
	(u_{-r}v)\,R^n_{m,k}z
	\;\equiv_{\,U\,}\;
	\Bigl((u_{-r}v)R^n_{m,k}z\Bigr)_{\mathrm{p}}
	\qquad(m\ge k\ge 0).
	\end{equation}
	By \cite[Lemma~4.16]{Part1}, after imposing the
	length-one relations and the singular quadratic iterates, the remaining
	length-two relations for a pair $\{u,v\}$ of homogeneous vectors are
	represented by the regular quadratic residues of the three ordered pair types
	\[
	u_{-r}u,\qquad u_{-r}v,\qquad v_{-r}v
	\qquad (r\ge1).
	\]
	This section studies one ordered pair $(u,v)$ at a
	time, modulo $U$; the equal-input case is the specialization $u=v$.
	All congruences $A\equiv_U B$ below are inclusions $A-B\in U$ in $V$
	(Definition~\ref{def:W-congruence}), and equalities are equalities in $V$.
\end{rem}

In particular, by Remark~\ref{rem:bridge-true-regular-quadratic} and the binomial
reduction \eqref{eq:Rmk-binomial-reduction}, the remaining quadratic problem for
$O_n^\circ(2,S,V)$ reduces to analyzing the product residues in the target case $k=0$:
\[
(u_{-r}v)\,R^n_{m,0}z
\;\equiv_{\,U\,}\;
\Bigl((u_{-r}v)R^n_{m,0}z\Bigr)_{\mathrm{p}}
\qquad(r\ge 1,\ m\ge 0,\ z\in V).
\]
The goal of this section is to compute and reduce the product contribution
$\bigl((u_{-r}v)R^n_{m,k}z\bigr)_{\mathrm{p}}$ using only modes with indices at least $-2n-1$, and then show that only
finitely many such product relations are needed (for fixed $n$) to generate the quadratic part.
The following specialization of \cite[Lemma~4.12]{Part1} performs the last-mode reduction; it is also used in the equal-input calculation of Section~\ref{subsec:equal-input-specialization}.

\begin{lem}[Expanding a two-iterate vacuum word against $R^n_{m,k}$]
\label{lem:two-iterate-word-Rmk-in-generators}
	Let $u,v\in V$ be homogeneous and let $z\in V$.  Fix $n\in\N$,
	$m\ge k\ge0$, $q\in\Z$, and $r\ge1$.  Then
	\begin{equation}\label{eq:two-iterate-word-Rmk-in-generators}
	\begin{aligned}
	&(u_qv_{-r}\mathbf1)R^n_{m,k}z\\
	&\quad=
	\sum_{s=1}^{r}\binom{q}{s-1}
	\sum_{h=0}^{r-s}
	\Biggl(
	\sum_{a=0}^{r-s-h}(-1)^a
	\binom{n+k+\wt u+\wt v-q+r-2}{a}\\
	&\qquad\qquad\qquad\qquad\cdot
	\binom{m+2n+r-s+1-a}{r-s-a}
	\binom{r-s-a}{h}
	\Biggr)
	(u_{q-s+1}v)R^n_{m+h,k}z.
	\end{aligned}
	\end{equation}
\end{lem}

\begin{proof}
	This is the $N=2$ specialization of \cite[Lemma~4.12]{Part1}, with
	$u^{(1)}=u$, $u^{(2)}=v$, $q_1=q$, and $q_2=r$, followed by the
	changes of variables $s=j+1$ and $h=m'-m$.
\end{proof}

For the rest of this section, congruences are taken modulo the local
subspace $U$ in \eqref{eq:U-def}. The pair space is
$\mathcal Q_n((u,v))$ from \eqref{eq:intro-Q-def}. In the last-mode reduction of
Lemma~\ref{lem:two-iterate-word-Rmk-in-generators}, if $q\ge0$, each surviving term has
$q-s+1\ge0$, so its residue belongs to a singular-iterate slice
already included in $U$. If $q<0$, the surviving inputs are
regular iterates. Consequently the pair space is spanned by
\begin{equation}\label{eq:pair-regular-residue-maps}
z\longmapsto (u_{-r}v)R^n_{m,k}z+U,
\qquad r\ge1,\quad m\ge k\ge0.
\end{equation}
Equality of two such induced maps means that their difference takes
every $z\in V$ into $U$, as in Definition~\ref{def:W-congruence}\textup{(3)}.

\begin{prop}[General $(m,k)$ product coefficient expansion]\label{prop:prod-master-general-mk}
	With the product contribution \eqref{eq:prod-piece-def-mk}, for all $m\ge k\ge0$, $r\ge 1$,
	homogeneous $u,v\in V$, and $z\in V$, there exist coefficients
	\[
	C^{(n)}_{m,k}(r;a,b)\in\C
	\qquad(a\ge -2n-1,\ b\ge -2n-1)
	\]
	such that the \emph{product contribution} admits the mode expansion
	\begin{equation}\label{eq:prod-C-expansion-mk}
	\Bigl((u_{-r}v)R^n_{m,k}z\Bigr)_{\mathrm{p}}
	=
	\sum_{a\ge -2n-1}\ \sum_{b\ge -2n-1}
	C^{(n)}_{m,k}(r;a,b)\;v_a\,u_b\,z.
	\end{equation}
	Moreover the coefficients are given by the $\Phi$--coefficient extraction formula
	\begin{equation}\label{eq:Gamma-discrete2}
	C^{(n)}_{m,k}(r;a,b)
	=
	\bigl[x^{m+2n+1}\bigr]\Bigl((1+x)^{k-n}\,\Phi_{a,1,v}(x)\,\Phi_{b,r,u}(x)\Bigr).
	\end{equation}
	By Corollary~\ref{cor:Phi-Laurent-and-shifted-poly}, $\Phi_{a,1,v}(x)$ and
	$\Phi_{b,r,u}(x)$ are Laurent polynomials, so each coefficient is a finite sum.
\end{prop}

\begin{proof}
	By \eqref{eq:YL-pform-def} and \eqref{eq:descendant-derivative-convention},
	the two truncated fields in \eqref{eq:prod-piece-def-mk} are
	\[
	Y_L(v,x)=\sum_{a\ge -2n-1} p_{a,1,v}(x)\,v_a,
	\qquad
	Y_L(u_{-r}\mathbf1,x)=\sum_{b\ge -2n-1} p_{b,r,u}(x)\,u_b.
	\]
	Applied to $z$, both sums are finite, so the residue may be taken termwise; this gives
	\eqref{eq:prod-C-expansion-mk} with
	\[
	C^{(n)}_{m,k}(r;a,b)
	=
	\Res_x\;
	\frac{(1+x)^{\wt u+\wt v+r-1+n+k}}{x^{m+2n+2}}\;
	p_{a,1,v}(x)\,p_{b,r,u}(x).
	\]
	By \eqref{eq:p-from-Phi-r},
	$p_{a,1,v}(x)=(1+x)^{-(\wt v+n)}\Phi_{a,1,v}(x)$ and
	$p_{b,r,u}(x)=(1+x)^{-(\wt u+n+r-1)}\Phi_{b,r,u}(x)$. The weights cancel
	from the power of $1+x$, which becomes $(1+x)^{k-n}$, and this is
	\eqref{eq:Gamma-discrete2}.
\end{proof}

\subsection{The universal coefficient space}
\label{subsec:formal-independence}

Throughout the quadratic coefficient calculations, fix an integer
\(n\ge1\) and homogeneous \(u,v\in V\).  All binomial coefficients use the convention in
\eqref{eq:how-to-expand}.  Powers of \(1+x\) are expanded at \(x=0\).
The retained modes have indices at least $-2n-1$, as in \eqref{eq:YL-pform-def}.
The coefficient arguments below use only Definition~\ref{def:Phi-jru} and
\eqref{eq:Gamma-discrete2}.

\begin{defn}[Universal regular--regular coefficient space]
	\label{def:universal-prod-reg-space}
	Define
	\[
	\mathcal F_n(u,v):=
	\bigoplus_{a=-2n-1}^{-1}\bigoplus_{b=-2n-1}^{-1}\C[v_a u_b].
	\]
	Here the symbols \([v_a u_b]\) are formal basis vectors.  For \(m\ge k\ge0\), \(r\ge1\), set
	\begin{equation}\label{eq:Gamma-prod-reg-def}
	\begin{aligned}
	\Gamma^{(n)}_{m,k,r}(u,v)
	&:=\sum_{a=-2n-1}^{-1}\sum_{b=-2n-1}^{-1}
	C^{(n)}_{m,k}(r;a,b)[v_a u_b],
	\end{aligned}
	\end{equation}
	where the coefficients are those specified by \eqref{eq:Gamma-discrete2}.
	Define
	\begin{equation}\label{eq:formal-prod-reg-span}
	\mathcal S_{\mathrm{prod,reg}}(n;u,v)
	:=\Span_\C\{\Gamma^{(n)}_{m,k,r}(u,v):m\ge k\ge0,\ r\ge1\}.
	\end{equation}
\end{defn}

The operator families of the introduction are
\begin{align}
\mathrm{(A)}&:=\{(u_{-r}v)R^n_{m,0}(\cdot):
n\le m\le2n-1,\ 1\le r\le2n-m\},\label{eq:A-strip}\\
\mathrm{(B)}&:=\{(u_{-r}v)R^n_{m,n}(\cdot):
n\le m\le2n-1,\ 1\le r\le2n-m\},\label{eq:B-strip}\\
\mathrm{(C)}&:=\{(u_{-r}v)R^n_{m,0}(\cdot):
0\le m\le n-1,\ 1\le r\le n-m\}.\label{eq:C-strip}
\end{align}
Each list has \(\binom{n+1}{2}\) members. The formal seed families
\(\mathcal A\), \(\mathcal B\), \(\mathcal C\) consist of the coefficient vectors of
their members:
\begin{equation}\label{eq:formal-ABC-seeds}
\mathcal A,\ \mathcal B,\ \mathcal C:=
\bigl\{\Gamma^{(n)}_{m,k,r}(u,v):\ (u_{-r}v)R^n_{m,k}(\cdot)\ \text{in}\
\mathrm{(A)},\ \mathrm{(B)},\ \mathrm{(C)},\ \text{respectively}\bigr\}.
\end{equation}

The finite vectors \(\Gamma^{(n)}_{m,k,r}(u,v)\) record only the regular
negative modes. To obtain identities between product operators, we must also
control the full scalar arrays
\[
\bigl(C^{(n)}_{m,k}(r;a,b)\bigr)_{a,b\ge-2n-1}.
\]
We regard these arrays as functions of \((a,b)\), allowing infinitely many
nonzero entries, and take their ordinary complex linear span, whose elements
are finite linear combinations of arrays. Restriction to
\(-2n-1\le a,b\le-1\) gives the finite vectors in
Definition~\ref{def:universal-prod-reg-space}. The next lemmas control the
entire arrays without introducing an additional coefficient-space notation.

\subsection{Formal linear independence and the three quadratic bases}
\label{subsec:formal-three-bases}

\begin{lem}[Weight removal on all retained modes]
\label{lem:formal-weight-removal}
For every \(p,q\ge0\), define a row functional on scalar arrays \(c(a,b)\),
\(a,b\ge-2n-1\), by
\begin{equation}\label{eq:formal-rho}
\rho_{p,q}(c):=
\sum_{\alpha=0}^{p}\sum_{\beta=0}^{q}
\binom{-\wt v-n}{p-\alpha}\binom{-\wt u-n}{q-\beta}
c(-2n-1+\alpha,-2n-1+\beta).
\end{equation}
When \(p,q\le2n\), this also defines a functional on \(\mathcal F_n(u,v)\)
by reading \(c(a,b)\) as the coefficient of \([v_a u_b]\).
For \(m\ge k\ge0\) and \(s\ge0\), retain the finite column combinations
\begin{equation}\label{eq:formal-hat-Gamma}
\widehat\Gamma^{(n)}_{m,k,s}(u,v):=
\sum_{r=1}^{s+1}\binom{\wt u+n}{s+1-r}
\Gamma^{(n)}_{m,k,r}(u,v).
\end{equation}
The full arrays satisfy the exact identity
\begin{equation}\label{eq:formal-weight-free-entry}
\begin{aligned}
&\sum_{r=1}^{s+1}\binom{\wt u+n}{s+1-r}
\rho_{p,q}\bigl(C^{(n)}_{m,k}(r;\cdot,\cdot)\bigr)\\
&\qquad=
\binom{2n-q}{s}\binom{k-n+s}{m+p+q+s-2n+1}.
\end{aligned}
\end{equation}
For \(p,q\le2n\), its left side is
\(\rho_{p,q}(\widehat\Gamma^{(n)}_{m,k,s}(u,v))\).
The row transformation is invertible on the full arrays and unitriangular
on every finite downward-closed set of rows. For each fixed \((m,k)\),
the column transformation is unitriangular on every initial list
\(r=1,\ldots,R\).
\end{lem}

\begin{proof}
The weights enter \(C^{(n)}_{m,k}(r;a,b)\) only through the factors
\(\binom{\wt v+n}{\cdot}\), \(\binom{\wt u+n}{\cdot}\) and
\(\binom{-\wt u-n}{\cdot}\) in \(\Phi\). For the lowest retained
modes, the numerator of \(\Phi\) is a partial sum of the binomial series of
\((1+x)^{\wt u+n}\) (respectively \((1+x)^{\wt v+n}\)). Convolving the rows
with the coefficients of the inverse series leaves a single monomial, and
the column convolution in \(r\) cancels the remaining factor
\(\binom{-\wt u-n}{\cdot}\).

The finite convolutions in Definition~\ref{def:Phi-jru} give, for all
\(p,q\ge0\),
\begin{align*}
\sum_{\alpha=0}^{p}\binom{-\wt v-n}{p-\alpha}
\Phi_{-2n-1+\alpha,1,v}(x)
&=x^{2n-p},\\
\sum_{\beta=0}^{q}\binom{-\wt u-n}{q-\beta}
\Phi_{-2n-1+\beta,r,u}(x)
&=\sum_{h=0}^{r-1}\binom{-\wt u-n}{r-1-h}
\binom{2n-q}{h}x^{2n-q-h}(1+x)^h.
\end{align*}
To see the cancellation in the second formula directly, write a numerator
summation index as \(\ell\), so \(0\le\ell\le\beta\), and put
\(d=\beta-\ell\). The coefficient of
\(\binom{2n-d}{h}x^{2n-d-h}(1+x)^h\) contains the finite sum
\[
\sum_{\ell=0}^{q-d}
\binom{-\wt u-n}{q-d-\ell}\binom{\wt u+n}{\ell}
=\delta_{q,d}.
\]
The first formula is the same calculation with \(r=1\). No bound
\(p,q\le2n\) is used.

Insert these formulas into \eqref{eq:Gamma-discrete2}. The subsequent
convolution in \(r\) cancels
\(\binom{-\wt u-n}{r-1-h}\), leaving
\[
\binom{2n-q}{s}[x^{m+2n+1}]
\bigl(x^{4n-p-q-s}(1+x)^{k-n+s}\bigr),
\]
which is \eqref{eq:formal-weight-free-entry}.
Each row uses finitely many predecessors and has diagonal coefficient one.
Its inverse is the same finite convolution with
\(\binom{\wt v+n}{p-\alpha}\binom{\wt u+n}{q-\beta}\).
Thus it is an automorphism of the space of all scalar arrays. On a finite
downward-closed set all predecessors remain in that set, giving the stated
unitriangular row change. The column change also has diagonal coefficient
one and uses only smaller descendant orders.
\end{proof}

\begin{lem}[An upper bound for the full coefficient arrays]
\label{lem:formal-dimension-upper-bound}
The complex span of all scalar arrays
\(C^{(n)}_{m,k}(r;\cdot,\cdot)\), with \(m\ge k\ge0\) and \(r\ge1\),
has dimension at most \(n(n+1)\). In particular,
\[
\dim\mathcal S_{\mathrm{prod,reg}}(n;u,v)\le n(n+1).
\]
\end{lem}

\begin{proof}
After reducing to \(k=0\) and applying the changes of
Lemma~\ref{lem:formal-weight-removal}, the entries on the rows with
\(p+q\ge n\) are values of polynomials in \((p,q)\) of degree at most
\(n-1\). These rows therefore contribute at most \(\binom{n+1}{2}\)
dimensions, and there are only \(\binom{n+1}{2}\) other rows.

Expanding the finite factor \((1+x)^k\) in
\eqref{eq:Gamma-discrete2} gives the same identity on every mode pair:
\[
C^{(n)}_{m,k}(r;a,b)=
\sum_{\ell=0}^{k}\binom{k}{\ell}C^{(n)}_{m-\ell,0}(r;a,b).
\]
Since \(m\ge k\), every index \(m-\ell\) is nonnegative. Restricting to
the regular square gives
\begin{equation}\label{eq:formal-k-reduction}
\Gamma^{(n)}_{m,k,r}(u,v)=
\sum_{\ell=0}^{k}\binom{k}{\ell}
\Gamma^{(n)}_{m-\ell,0,r}(u,v).
\end{equation}
It therefore suffices to bound the full arrays with \(k=0\).
The finite invertible column changes in
\eqref{eq:formal-hat-Gamma} allow us to use the corresponding arrays
indexed by \(s\ge0\), and the invertible row change of
Lemma~\ref{lem:formal-weight-removal} preserves their span dimension.

On every row \(p+q\ge n\), the transformed entry
\eqref{eq:formal-weight-free-entry} vanishes for \(s\ge n\): its second
lower index exceeds the nonnegative upper index by
\(m+p+q-n+1>0\). For \(0\le s<n\), that entry equals
\begin{equation}\label{eq:formal-high-polynomial}
(-1)^{m+p+q+s-2n+1}
\binom{2n-q}{s}\binom{m+p+q-n}{n-s-1}.
\end{equation}
The negative-binomial identity proves this when the extracted index is
nonnegative. If it is negative, both sides vanish, since
\(m+p+q-n\ge0\) and it is then smaller than \(n-s-1\).
After multiplying each row by \((-1)^{p+q}\), the entries on these rows
are evaluations of polynomials in \(p,q\) of total degree at most
\(n-1\). This remains true for \(q>2n\), using generalized binomial
coefficients. Their high-row restrictions consequently span a space of
dimension at most \(\binom{n+1}{2}\).

There are exactly \(\binom{n+1}{2}\) low rows \(p+q\le n-1\).
Restriction to the low and high rows together is injective, so the
dimension of the whole transformed span is at most the sum of these two
bounds, namely \(n(n+1)\). Restriction to the regular square cannot
increase dimension, proving the final assertion.
\end{proof}

\begin{lem}[Polynomial bases and triangular grids]
	\label{lem:formal-polynomial-grids}
	For any \(c\in\C\), the polynomials
	\begin{equation}\label{eq:formal-polynomial-basis}
	\binom{2n-t+p}{s}\binom{c+t+i}{n-s-1},
	\qquad i,s\ge0,\quad i+s\le n-1,
	\end{equation}
	form a basis of the polynomials in \(t,p\) of total degree at most \(n-1\).
	Evaluation of such polynomials is injective on either grid
	\begin{equation}\label{eq:formal-low-high-grids}
	\begin{aligned}
	\mathcal L_n&:=\{(t,p):0\le t\le n-1,\ 0\le p\le t\},\\
	\mathcal H_n&:=\{(t,p):n\le t\le2n-1,\ 0\le p\le2n-1-t\}.
	\end{aligned}
	\end{equation}
\end{lem}

\begin{proof}
	For fixed \(s\), the \(n-s\) polynomials
	\(\binom{c+t+i}{n-s-1}\), \(0\le i\le n-s-1\), form a basis in
	one variable of degree at most \(n-s-1\), since finite differences give
	\[
	\sum_{i=0}^{j}(-1)^{j-i}\binom ji
	\binom{c+t+i}{n-s-1}=\binom{c+t}{n-s-1-j}
	\quad(0\le j\le n-s-1).
	\]
	In a linear relation among \eqref{eq:formal-polynomial-basis}, take the
	largest \(s\) with a nonzero coefficient polynomial in \(t\).
	Its coefficient of \(p^s\) is that polynomial divided by \(s!\), a
	contradiction.  The number of polynomials is \(\binom{n+1}{2}\), proving
	the basis claim.  On the grid \(j,p\ge0\), \(j+p\le n-1\), the evaluation
	matrix of \(\binom ja\binom pb\), \(a+b\le n-1\), is unitriangular
	when rows and columns are ordered by total degree: equal total degrees
	give an identity block.  Therefore this grid is unisolvent.  Apply the
	invertible affine substitutions \(j=t-p\) and \(j=t-n\) to obtain
	\(\mathcal L_n\) and \(\mathcal H_n\), respectively.
\end{proof}

\begin{prop}[Three bases of the universal coefficient space]
	\label{prop:AB-prodreg-full-rank}
	Each of the indexed unions
	\[
	\mathcal A\cup\mathcal B,\qquad
	\mathcal C\cup\mathcal B,\qquad
	\mathcal A\cup\mathcal C
	\]
	is a basis of \(\mathcal S_{\mathrm{prod,reg}}(n;u,v)\).  In particular,
	\begin{equation}\label{eq:formal-exact-dimension}
	\dim\mathcal S_{\mathrm{prod,reg}}(n;u,v)=n(n+1).
	\end{equation}
The corresponding complete scalar arrays are bases of their full
complex span. Restriction to the regular mode square is an isomorphism
from that span to $\mathcal S_{\mathrm{prod,reg}}(n;u,v)$.
\end{prop}

\begin{proof}
	By Lemma~\ref{lem:formal-dimension-upper-bound}, it suffices to show that
	each union is linearly independent. We test this on \(n(n+1)\) rows on
	which, after the changes of Lemma~\ref{lem:formal-weight-removal}, each
	matrix is block triangular and every diagonal block is an invertible
	polynomial evaluation matrix.

	Use the \(n(n+1)\) rows
	\[
	\mathcal W_n:=\{(p,q):0\le p\le n-1,\ 0\le q\le2n-2p-1\}.
	\]
	This window is downward closed, so \eqref{eq:formal-rho} is an invertible
	row operation within this window.  At each fixed \((m,k)\), each seed
	list is an initial \(r\)-list, so \eqref{eq:formal-hat-Gamma} is also an
	invertible column operation.  Put \(t=p+q\).  The window is precisely
	the disjoint union \(\mathcal L_n\cup\mathcal H_n\) of
	\eqref{eq:formal-low-high-grids}.

	Index the transformed columns by \(i,s\ge0\), \(i+s\le n-1\), where
	\(s=r-1\), and \(i=m-n\) for \(\mathcal A\) and \(\mathcal B\), \(i=m\) for \(\mathcal C\).
	For \(\mathcal A,\mathcal B,\mathcal C\) their entries, respectively,
	are the following scalars, by \eqref{eq:formal-weight-free-entry}:
	\begin{equation}\label{eq:formal-ABC-matrix-entries}
	\begin{aligned}
	\mathcal A:\quad &\binom{2n-t+p}{s}\binom{s-n}{i+t+s-n+1},\\
	\mathcal B:\quad &\binom{2n-t+p}{s}\binom{s}{n-1-i-t},\\
	\mathcal C:\quad &\binom{2n-t+p}{s}\binom{s-n}{i+t+s-2n+1}.
	\end{aligned}
	\end{equation}
	The symmetry used in the middle expression is valid because \(s\ge0\).

	\smallskip
	\noindent\emph{The zero blocks.}
	On \(\mathcal H_n\), the \(\mathcal B\)-entry is zero because
	\(n-1-i-t<0\).  On \(\mathcal L_n\), the \(\mathcal C\)-entry is zero
	because \(i+s\le n-1\) gives \(i+t+s-2n+1\le t-n<0\).

	\smallskip
	\noindent\emph{The two \(\mathcal A\)-blocks and the high \(\mathcal C\)-block.}
	The first expression in \eqref{eq:formal-ABC-matrix-entries}, on either
	grid, is
	\[
	(-1)^{i+t+s-n+1}
	\binom{2n-t+p}{s}\binom{i+t}{n-s-1}.
	\]
	The third expression, on \(\mathcal H_n\), is
	\[
	(-1)^{i+t+s-2n+1}
	\binom{2n-t+p}{s}\binom{i+t-n}{n-s-1}.
	\]
	In each case the sign splits into a row sign and a column sign.
	The upper indices \(i+t\) and \(i+t-n\) are nonnegative on their
	respective grids, so the formulas also hold when the original extracted
	index is negative, with both sides zero.  Lemma~\ref{lem:formal-polynomial-grids},
	with \(c=0\) or \(c=-n\), proves that all three square blocks are invertible.

	\smallskip
	\noindent\emph{The low \(\mathcal B\)-block.}
	Reindex its columns by \(h=n-1-i\), so \(0\le s\le h\le n-1\).
	Its entry is
	\[
	\binom{2n-t+p}{s}\binom{s}{h-t}.
	\]
	It is zero if \(t>h\).  Ordering the row blocks by \(t\) and the column
	blocks by \(h\) makes this a block-triangular matrix.  The diagonal block
	\(t=h\) is \((\binom{2n-t+p}{s})_{0\le p,s\le t}\).
	Its columns are a polynomial basis of degrees \(0,\ldots,t\), evaluated
	at the distinct points \(p=0,\ldots,t\); hence it is invertible.

	Consequently the \(\mathcal A\cup\mathcal B\) and
	\(\mathcal C\cup\mathcal B\) matrices are block triangular on the row
	order \(\mathcal H_n,\mathcal L_n\), with invertible diagonal blocks.
	The \(\mathcal A\cup\mathcal C\) matrix has the same property on the row
	order \(\mathcal L_n,\mathcal H_n\).  Thus all three original matrices have full rank. Each union contains
	\(n(n+1)\) independent finite vectors, and its full scalar arrays are
	therefore independent as well. Lemma~\ref{lem:formal-dimension-upper-bound}
	gives the matching upper bound on the full-array span. Consequently each
	union is a basis both before and after restriction to the regular square.
	Restriction takes a basis to a basis, proving that it is an isomorphism.
\end{proof}

\subsection{Spanning the quadratic residue operators}

The scalar bases now give identities of product operators. The use of full
arrays is essential here: equality of coefficients only on the regular
negative-mode square would not control the remaining terms in
\eqref{eq:prod-C-expansion-mk}.

\begin{prop}[Spanning set for \(\mathcal Q_n((u,v))\)]
\label{prop:empirical-rank}
The images of the operator families \(\mathrm{(A)}\) and
\(\mathrm{(B)}\) in \eqref{eq:A-strip}--\eqref{eq:B-strip} span
\(\mathcal Q_n((u,v))\). In particular,
\[
\dim\mathcal Q_n((u,v))\le n(n+1).
\]
The number of listed operators is
\[
\#\mathrm{(A)}+\#\mathrm{(B)}
=2\sum_{m=n}^{2n-1}(2n-m)=n(n+1).
\]
All reduction coefficients are independent of the right input \(z\).
\end{prop}

\begin{proof}
By Proposition~\ref{prop:AB-prodreg-full-rank}, every full scalar array
\(C^{(n)}_{m,k}(r;\cdot,\cdot)\) is a finite linear combination of the
arrays with indices in \(\mathcal A\cup\mathcal B\). These scalar
coefficients depend on the fixed pair \((u,v)\) and on \(n,m,k,r\),
and do not depend on \(z\).

Multiply the coefficient identity at \((a,b)\) by \(v_a u_bz\) and sum
over all \(a,b\ge-2n-1\). On a fixed \(z\), only finitely many \(b\)
give \(u_bz\ne0\), and for each of them only finitely many \(a\)
give \(v_a u_bz\ne0\). Thus lower truncation makes every sum finite.
Equation~\eqref{eq:prod-C-expansion-mk} gives the corresponding exact
linear combination of product contributions in \(V\).

By \eqref{eq:drop-comm-strip}, each product contribution is congruent
modulo the same local subspace \(U\) to its full residue. We obtain the
claimed linear combination of residue operators modulo \(U\), valid
on every \(z\in V\). Finally, the induced maps
\eqref{eq:pair-regular-residue-maps} span \(\mathcal Q_n((u,v))\),
so the seed operators span that space as well. Counting them gives the
dimension bound.
\end{proof}

The coefficients in Proposition~\ref{prop:empirical-rank} are obtained by
inverting finite matrices. Explicit recurrences that reduce every product
coefficient array to the seeds of $\mathrm{(A)}\cup\mathrm{(B)}$, with closed-form
coefficients, and thereby give a second proof of
Lemma~\ref{lem:formal-dimension-upper-bound}, are recorded in the ancillary
file \texttt{anc/\allowbreak explicit-\allowbreak quadratic-\allowbreak recurrences.pdf}.

\begin{rem}[The operator bound can be strict for minimal strong generators]
	\label{rem:affine-level-one-rank-drop}
	Let \(V=L_1(\mathfrak{sl}_2)\), with the usual weight-one states
	\(e,h,f\).  The set \(\{e,h,f\}\) is a minimal homogeneous strong
	generating set.  Indeed, it strongly generates \(V\), while every
	nonvacuum normally ordered word of length at least two and every positive
	derivative has weight at least two.  Hence a homogeneous strong generating
	set must span \(V_1\cong\mathfrak{sl}_2\), so three weight-one generators
	are necessary.
	The level-one singular-vector relation is
	\[
	e_{-1}e=0.
	\]
	Using \([e,f]=h\) and \([h,e]=2e\), application of \(f_0\) gives
	\[
	0=f_0(e_{-1}e)=-h_{-1}e-e_{-1}h.
	\]
	Since \([h_{-1},e_{-1}]=2e_{-2}\), it follows that
	\begin{equation}\label{eq:affine-level-one-descendant-relation}
	e_{-1}h=-e_{-2}\mathbf1=-L(-1)e.
	\end{equation}
	Consequently, for every \(m\ge k\),
	\[
	(e_{-1}h)R^n_{m,k}(\cdot)
	=-(e_{-2}\mathbf1)R^n_{m,k}(\cdot)
	\]
	has image in \(O_n^\circ(e,V)\) by
	Lemma~\ref{lem:descendant-circle-subspace-inclusion}, and hence represents
	zero in \(\mathcal Q_n((e,h))\).

	Each of \(\mathrm{(A)}\) and \(\mathrm{(B)}\) contains exactly \(n\)
	members with \(r=1\).  Deleting these \(2n\) zero members from the
	\(n(n+1)\)-element spanning list gives the strict estimate
	\begin{equation}\label{eq:affine-level-one-rank-drop}
	\dim\mathcal Q_n((e,h))\le n(n-1)<n(n+1).
	\end{equation}
	We do not determine the exact dimension of
	\(\mathcal Q_n((e,h))\).
\end{rem}

\subsection{The specialization \texorpdfstring{$u=v$}{u = v}}
\label{subsec:equal-input-specialization}

We now specialize the ordered pair to \((u,u)\).  In this case the two
length-one summands of \(U\) coincide, so throughout this subsection
the specialization of \eqref{eq:U-def} is
\[
U=O_n^\circ(u,V)+\sum_{j\ge0}O_n^\circ(u_j u,V).
\]

The reduction comes from interchanging the two modes in the vacuum word.  The
following triangular form is obtained by applying
Lemma~\ref{lem:two-iterate-word-Rmk-in-generators} with \(q=-1\) and \(v=u\).

\begin{lem}[Interchanging the two modes]
\label{lem:equal-input-mode-interchange}
Let \(m\ge k\ge0\) and \(r\ge1\).  Then
\begin{equation}\label{eq:equal-input-swapped-expansion}
\begin{aligned}
&\bigl(u_{-1}u_{-r}\mathbf1\bigr)R^n_{m,k}(\cdot)\\
&\qquad=(-1)^{r-1}(u_{-r}u)R^n_{m,k}(\cdot)\\
&\qquad\quad+\sum_{s=1}^{r-1}\sum_{h=0}^{r-s}
(-1)^{s-1}
\Biggl(
\sum_{a=0}^{r-s-h}(-1)^a
\binom{n+k+2\wt u+r-1}{a}\\
&\hspace{38mm}\cdot
\binom{m+2n+r-s+1-a}{r-s-a}
\binom{r-s-a}{h}
\Biggr)(u_{-s}u)R^n_{m+h,k}(\cdot).
\end{aligned}
\end{equation}
In particular, every term in the sum has \(s<r\) and satisfies
\[
(m+h)+s\le m+r.
\]
\end{lem}

\begin{proof}
Apply Lemma~\ref{lem:two-iterate-word-Rmk-in-generators} with \(q=-1\)
and \(v=u\).  Since \(\binom{-1}{s-1}=(-1)^{s-1}\), its displayed
coefficient is exactly the coefficient in
\eqref{eq:equal-input-swapped-expansion}.  The only term with \(s=r\)
has \(h=0\), and its coefficient is \((-1)^{r-1}\).  All remaining terms
have \(s<r\), and the summation range gives
\((m+h)+s\le m+r\).
\end{proof}

For example, the first nontrivial case is
\begin{equation}\label{eq:equal-input-r2-example}
\begin{aligned}
\bigl(u_{-1}u_{-2}\mathbf1\bigr)R^n_{m,k}(\cdot)
={}&-(u_{-2}u)R^n_{m,k}(\cdot)\\
& +(m+n-k-2\wt u+1)(u_{-1}u)R^n_{m,k}(\cdot)\\
& +(m+2n+2)(u_{-1}u)R^n_{m+1,k}(\cdot).
\end{aligned}
\end{equation}

\begin{prop}[Odd-mode spanning sets when \(u=v\)]
\label{prop:equal-input-odd-spanning}
For homogeneous \(u\in V\), the following two families span
\(\mathcal Q_n((u,u))\):
\begin{align}
\mathrm{(A)}_{\mathrm{odd}}
&:=\left\{
(u_{-r}u)R^n_{m,0}(\cdot):
n\le m\le2n-1,\quad
1\le r\le2n-m,\quad r\ \text{odd}
\right\},
\label{eq:A-odd-strip}\\
\mathrm{(B)}_{\mathrm{odd}}
&:=\left\{
(u_{-r}u)R^n_{m,n}(\cdot):
n\le m\le2n-1,\quad
1\le r\le2n-m,\quad r\ \text{odd}
\right\}.
\label{eq:B-odd-strip}
\end{align}
Each family has
\[
\left\lfloor\frac{(n+1)^2}{4}\right\rfloor
\]
members.  Consequently,
\begin{equation}\label{eq:equal-input-dimension-upper-bound}
\dim\mathcal Q_n((u,u))
\le
2\left\lfloor\frac{(n+1)^2}{4}\right\rfloor.
\end{equation}
\end{prop}

\begin{proof}
The mode commutator formula gives
\[
u_{-r}u_{-1}\mathbf1-u_{-1}u_{-r}\mathbf1
=[u_{-r},u_{-1}]\mathbf1
=\sum_{j\ge0}\binom{-r}{j}
(u_j u)_{-r-j-1}\mathbf1.
\]
Every state on the right is a vacuum descendant of a singular iterate
\(u_j u\).  Lemma~\ref{lem:descendant-circle-subspace-inclusion} therefore
shows that, after applying any \(R^n_{m,k}\), its value lies in \(U\).  Hence
\begin{equation}\label{eq:equal-input-commute-mod-U}
(u_{-r}u)R^n_{m,k}(\cdot)
\equiv_U
\bigl(u_{-1}u_{-r}\mathbf1\bigr)R^n_{m,k}(\cdot).
\end{equation}
Combining this congruence with
\eqref{eq:equal-input-swapped-expansion} gives
\begin{equation}\label{eq:equal-input-parity-reduction}
\begin{aligned}
&\bigl(1-(-1)^{r-1}\bigr)(u_{-r}u)R^n_{m,k}(\cdot)\\
&\qquad\equiv_U
\sum_{s=1}^{r-1}\sum_{h=0}^{r-s}
(-1)^{s-1}
\Biggl(
\sum_{a=0}^{r-s-h}(-1)^a
\binom{n+k+2\wt u+r-1}{a}\\
&\hspace{38mm}\cdot
\binom{m+2n+r-s+1-a}{r-s-a}
\binom{r-s-a}{h}
\Biggr)(u_{-s}u)R^n_{m+h,k}(\cdot).
\end{aligned}
\end{equation}
If \(r\) is even, the coefficient on the left is \(2\).  Thus every
even-\(r\) term is a linear combination, modulo \(U\), of terms with smaller
descendant order \(s<r\).

Suppose now that \(n\le m\le2n-1\) and \(m+r\le2n\).  Every term on the
right of \eqref{eq:equal-input-parity-reduction} satisfies
\[
m+h\ge n,
\qquad
(m+h)+s\le m+r\le2n,
\]
so it is again a member of the same family \(\mathrm{(A)}\) or \(\mathrm{(B)}\).
Induction on \(r\) therefore reduces family \(\mathrm{(A)}\) to
\(\mathrm{(A)}_{\mathrm{odd}}\), and family \(\mathrm{(B)}\) to
\(\mathrm{(B)}_{\mathrm{odd}}\).  Proposition~\ref{prop:empirical-rank}
then proves the spanning statement.

At fixed \(m\), there are \(\lceil(2n-m)/2\rceil\) odd
integers between \(1\) and \(2n-m\).  Hence each reduced family has
\[
\sum_{m=n}^{2n-1}\left\lceil\frac{2n-m}{2}\right\rceil
=\left\lfloor\frac{(n+1)^2}{4}\right\rfloor
\]
members, proving \eqref{eq:equal-input-dimension-upper-bound}.
\end{proof}

We next record what the same reduction does to the formal coefficient space.
This is the appropriate formal model for equal inputs; it is smaller than the
ordered model of Definition~\ref{def:universal-prod-reg-space}.

The formal version of interchange is simpler: it exchanges scalar
coefficient coordinates. On the ambient finite space \(\mathcal F_n(u,u)\),
define
\[
J([u_a u_b]):=[u_b u_a],\qquad -2n-1\le a,b\le-1.
\]
The following lemma identifies its
restriction to the coefficient space with the same triangular formula
used in \eqref{eq:equal-input-swapped-expansion}.

\begin{lem}[The interchange \(J\) is an involution]
\label{lem:equal-input-J-involution}
On the ambient space \(\mathcal F_n(u,u)\),
\begin{equation}\label{eq:equal-input-J-involution}
J^2=1.
\end{equation}
Moreover \(J\) preserves the regular coefficient space
$\mathcal S_{\mathrm{prod,reg}}(n;u,u)$ and, separately, the spans
of $\mathcal A$ and $\mathcal B$, and for \(m\ge k\ge0\) and \(r\ge1\) it acts by
\begin{equation}\label{eq:equal-input-J-def}
\begin{aligned}
&J\bigl(\Gamma^{(n)}_{m,k,r}(u,u)\bigr)\\
&\qquad=(-1)^{r-1}\Gamma^{(n)}_{m,k,r}(u,u)\\
&\qquad\quad+
\sum_{s=1}^{r-1}\sum_{h=0}^{r-s}(-1)^{s-1}
\Biggl(\sum_{j=0}^{r-s-h}(-1)^j
\binom{n+k+2\wt u+r-1}{j}\\
&\hspace{36mm}\cdot
\binom{m+2n+r-s+1-j}{r-s-j}
\binom{r-s-j}{h}\Biggr)
\Gamma^{(n)}_{m+h,k,s}(u,u).
\end{aligned}
\end{equation}
In particular \(J\) restricts to an involution of each of these subspaces.
\end{lem}

\begin{proof}
\emph{The involution.} \(J\) exchanges the basis symbols \([u_au_b]\) and
\([u_bu_a]\) of \(\mathcal F_n(u,u)\), so applying it twice is the identity.
This is \eqref{eq:equal-input-J-involution}; it holds on every invariant
subspace once invariance is shown. The rest of the proof establishes
\eqref{eq:equal-input-J-def}, from which invariance follows.

\emph{Step 1: the coefficients of \(J(\Gamma)\).} By
\eqref{eq:Gamma-prod-reg-def}, the coefficient of \([u_au_b]\) in
\(J(\Gamma^{(n)}_{m,k,r}(u,u))\) is \(C^{(n)}_{m,k}(r;b,a)\); the regular square
is symmetric under \(a\leftrightarrow b\). Write
\[
K_r(x):=\frac{(1+x)^{2\wt u+r-1+n+k}}{x^{m+2n+2}}.
\]
By the first display in the proof of
Proposition~\ref{prop:prod-master-general-mk}, with \(v=u\), and by
\eqref{eq:pjmu-def}, for every retained pair \(a,b\),
\[
C^{(n)}_{m,k}(r;b,a)
=\Res_x K_r(x)\,p_{b,1,u}(x)\,
\frac{\partial_x^{r-1}p_{a,1,u}(x)}{(r-1)!},
\]
whereas \(C^{(n)}_{m,k}(s;a,b)\) has the derivative of order \(s-1\) on
\(p_{b,1,u}\) instead. The proof moves the derivatives across.

\emph{Step 2: integration by parts.} The residue of a formal derivative of
a Laurent series is zero, so \(\Res_x f\,\partial_xg=-\Res_x(\partial_xf)\,g\).
Applied \(r-1\) times,
\[
C^{(n)}_{m,k}(r;b,a)
=\frac{(-1)^{r-1}}{(r-1)!}
\Res_x p_{a,1,u}(x)\,\partial_x^{r-1}
\bigl(K_r(x)p_{b,1,u}(x)\bigr).
\]

\emph{Step 3: Leibniz rule.} Since
\(\partial_x^{r-1}(fg)/(r-1)!=\sum_{s=1}^{r}\frac{\partial_x^{r-s}f}{(r-s)!}
\frac{\partial_x^{s-1}g}{(s-1)!}\),
\[
C^{(n)}_{m,k}(r;b,a)
=(-1)^{r-1}\sum_{s=1}^{r}\Res_x p_{a,1,u}(x)\,
\frac{\partial_x^{r-s}K_r(x)}{(r-s)!}\,
\frac{\partial_x^{s-1}p_{b,1,u}(x)}{(s-1)!}.
\]

\emph{Step 4: the kernel.} By Leibniz's rule and
\(\frac1{p!}\partial_x^px^{-B}=(-1)^p\binom{B+p-1}{p}x^{-B-p}\),
\[
\begin{aligned}
\frac{\partial_x^{r-s}K_r(x)}{(r-s)!}
={}&\sum_{j=0}^{r-s}(-1)^{r-s-j}
\binom{2\wt u+r-1+n+k}{j}
\binom{m+2n+r-s+1-j}{r-s-j}\\
&\qquad\cdot\frac{(1+x)^{2\wt u+r-1+n+k-j}}
{x^{m+2n+2+r-s-j}}.
\end{aligned}
\]
Since \(2\wt u+r-1+n+k-j=(2\wt u+s-1+n+k)+(r-s-j)\), the \(j\)-term is
\(K_s(x)\) times
\[
\left(\frac{1+x}{x}\right)^{r-s-j}
=\sum_{h=0}^{r-s-j}\binom{r-s-j}{h}x^{-h},
\]
and multiplying \(K_s\) by \(x^{-h}\) replaces \(m\) by \(m+h\). Each term
is therefore a coefficient \(C^{(n)}_{m+h,k}(s;a,b)\). The signs combine to
\((-1)^{r-1}(-1)^{r-s-j}=(-1)^{s-1+j}\). Exchanging the \(j\)- and
\(h\)-sums gives \eqref{eq:equal-input-J-def}, including its diagonal term
\(s=r\), \(h=j=0\). No assumption on the sign of \(\wt u\) is used.

\emph{Step 5: invariance.} Every term on the right of
\eqref{eq:equal-input-J-def} lies in the coefficient space, since
$m+h\ge k$. For a seed in \(\mathcal A\) or \(\mathcal B\), its indices
satisfy \(m\ge n\) and \(m+r\le2n\). The terms on the right satisfy
\(m+h\ge n\) and \((m+h)+s\le m+r\le2n\), and keep the same \(k\).
Thus the two seed spans are separately preserved.
\end{proof}

\begin{defn}[Coefficient space for equal inputs]
\label{def:equal-input-coefficient-space}
Define
\begin{equation}\label{eq:equal-input-coefficient-space}
\mathcal S_{\mathrm{eq}}(n;u)
:=
\mathcal S_{\mathrm{prod,reg}}(n;u,u)
\big/
\operatorname{im}(1-J).
\end{equation}
Thus the formal coefficient space for \(u=v\) imposes exactly the relations
forced by interchanging the two identical vacuum factors.
\end{defn}

\begin{prop}[Formal independence after setting \(u=v\)]
\label{prop:equal-input-formal-basis}
The images in \(\mathcal S_{\mathrm{eq}}(n;u)\) of the formal seeds indexed by
\(\mathrm{(A)}_{\mathrm{odd}}\cup\mathrm{(B)}_{\mathrm{odd}}\) form a basis.
In particular,
\begin{equation}\label{eq:equal-input-formal-dimension}
\dim\mathcal S_{\mathrm{eq}}(n;u)
=2\left\lfloor\frac{(n+1)^2}{4}\right\rfloor.
\end{equation}
\end{prop}

\begin{proof}
Order each of the bases \(\mathcal A\) and \(\mathcal B\) by increasing
descendant order \(r\).  Equation~\eqref{eq:equal-input-J-def} is triangular in
that order, with diagonal entry \((-1)^{r-1}\) at a seed of descendant order
\(r\).  Since \(J^2=1\), the map \(J\) is diagonalizable with eigenvalues
\(1\) and \(-1\).  Its \(1\)-eigenspace therefore has dimension equal to the
number of odd-\(r\) diagonal entries, namely
\[
2\left\lfloor\frac{(n+1)^2}{4}\right\rfloor.
\]
Moreover,
\[
\mathcal S_{\mathrm{prod,reg}}(n;u,u)
=\ker(1-J)\oplus\operatorname{im}(1-J),
\]
so the quotient in \eqref{eq:equal-input-coefficient-space} has that same
dimension.  The even-\(r\) reduction
\eqref{eq:equal-input-parity-reduction}, read formally, shows that the
odd-\(r\) classes span the quotient.  Their number equals its dimension, so
they are a basis.
\end{proof}

\begin{cor}[Global quadratic spanning and map count]
\label{cor:global-quadratic-bound}
Let $n\ge1$ and let $S$ be a finite homogeneous strong generating set,
with $\ell=|S|$. Fix an ordering of $S$. Modulo
$O_n^\circ(2,S,V)^+$, the space $O_n^\circ(2,S,V)$ is spanned by the
values of
\begin{equation}\label{eq:global-quadratic-spanning-list}
(s_{-r}t)R^n_{m,k}(\cdot),\qquad
s\le t,\quad n\le m\le2n-1,\quad1\le r\le2n-m,\quad k\in\{0,n\},
\end{equation}
with only odd $r$ needed when $s=t$. Equivalently, the space
\begin{equation}\label{eq:global-plus-induced-space}
\begin{aligned}
\mathcal Q_n(2,S):=\Span_\C\bigl\{&z\longmapsto wR^n_{m,k}z+O_n^\circ(2,S,V)^+:
\ w\in\mathcal P_{\le2}(S),\ m\ge k\ge0\bigr\}\\
&\subseteq\Hom\bigl(V,V/O_n^\circ(2,S,V)^+\bigr)
\end{aligned}
\end{equation}
is spanned by the maps from \eqref{eq:global-quadratic-spanning-list}, and
\begin{equation}\label{eq:global-quadratic-bound}
\dim\mathcal Q_n(2,S)\le\binom{\ell}{2}n(n+1)+2\ell\left\lfloor\frac{(n+1)^2}{4}\right\rfloor.
\end{equation}
This is also the number of listed operators. The reduction coefficients
are independent of the right input $z$.
\end{cor}

\begin{proof}
\cite[Lemma~4.16]{Part1} reduces to one cross-pair for each
unordered pair of distinct generators and one equal-input pair per
generator. For a pair $(s,t)$, its local subspace
\begin{equation}\label{eq:pair-modulus-in-global-plus}
O_n^\circ(s,V)+O_n^\circ(t,V)+\sum_{i\ge0}O_n^\circ(s_i t,V)
\subseteq O_n^\circ(2,S,V)^+
\end{equation}
consists of length-one slices and singular quadratic slices.
Postcomposition with the quotient by the plus subspace therefore
preserves every local pair reduction. Proposition~\ref{prop:empirical-rank}
gives $n(n+1)$ operators per cross-pair, and Proposition
\ref{prop:equal-input-odd-spanning} gives
$2\lfloor(n+1)^2/4\rfloor$ per equal-input pair.
There are $\binom{\ell}{2}$ cross-pairs and $\ell$ equal-input pairs,
which proves the count. The reductions in those propositions have
coefficients independent of $z$, so they prove spanning of the induced
maps as well as spanning of their evaluated values.
\end{proof}

\begin{rem}[Formal and operator statements]\label{rem:formal-versus-operator}
Propositions~\ref{prop:empirical-rank} and~\ref{prop:equal-input-odd-spanning} and
Corollary~\ref{cor:global-quadratic-bound} bound the span of the induced maps from $V$ into
the quotient, not the dimension of the quotient, since each map may have an
infinite-dimensional image.  Propositions~\ref{prop:AB-prodreg-full-rank}
and~\ref{prop:equal-input-formal-basis} concern formal coefficient spaces; the first supplies
the invertible matrices used in the eliminations below and in Part~III.  Formal independence
does not imply that the actual operators are independent
(Remark~\ref{rem:affine-level-one-rank-drop}).
\end{rem}

\subsection{Diagonal elimination for \texorpdfstring{$k=0$}{k = 0}}
\label{subsec:quadratic-diagonal-elimination}

We finish the quadratic analysis by recording two consequences of the
coefficient formulas and the formal rank theorem.  They are stated directly
in terms of the coefficients \(C^{(n)}_{m,0}(r;a,b)\) from
Proposition~\ref{prop:prod-master-general-mk}.

Recall from \eqref{eq:U-def} that, for the fixed ordered pair \((u,v)\),
\[
U=
O_n^\circ(u,V)+O_n^\circ(v,V)
+\sum_{i\ge0}O_n^\circ(u_i v,V).
\]
For use in this subsection, let
\begin{align}
U_A
&:=U+\Span_{\C}\left\{
(u_{-r}v)R^n_{n+p_0,0}z:
0\le p_0\le n-1,\ 1\le r\le n-p_0,\ z\in V
\right\},
\label{eq:UA-diagonal-elimination}\\
U_C
&:=U+\Span_{\C}\left\{
(u_{-r}v)R^n_{m,0}z:
0\le m\le n-1,\ 1\le r\le n-m,\ z\in V
\right\}.
\label{eq:UC-diagonal-elimination}
\end{align}
Thus the strip~\({\rm(A)}\) relations vanish modulo \(U_A\), and the
strip~\({\rm(C)}\) relations vanish modulo \(U_C\).

\begin{lem}[Strip A: deep diagonal elimination]
\label{lem:generaln-stripA-deep-elim-explicit}
Fix \(n\ge1\), homogeneous \(u,v\in V\), and \(z\in V\).  Let
\(K_A^{(n)}\) be the square matrix whose rows are indexed by
\[
(p_0,r),\qquad 0\le p_0\le n-1,\quad 1\le r\le n-p_0,
\]
and whose columns are indexed by
\[
(p,q),\qquad p,q\ge0,\quad p+q\le n-1,
\]
with entries
\begin{equation}\label{eq:KA-direct-C-entry}
K_A^{(n)}\bigl((p_0,r),(p,q)\bigr)
:=
C^{(n)}_{n+p_0,0}
\bigl(r;-2n-1+p,-2n-1+q\bigr).
\end{equation}
Then \(K_A^{(n)}\) is invertible.  For \(p+q\le n-1\) and
\(a,b\ge-2n-1\) with \(a+b\ge-3n-2\), set
\begin{equation}\label{eq:Epq-direct-C-formula}
\begin{aligned}
E^{(n)}_{p,q}(a,b)
:={}&-
\sum_{p_0=0}^{n-1}\ \sum_{r=1}^{n-p_0}
\bigl(K_A^{(n)}\bigr)^{-1}
\bigl((p,q),(p_0,r)\bigr)\\
&\hspace{34mm}\cdot C^{(n)}_{n+p_0,0}(r;a,b).
\end{aligned}
\end{equation}
Then
\begin{equation}\label{eq:deep-elim-generaln}
v_{-2n-1+p}u_{-2n-1+q}z
\equiv_{U_A}
\sum_{\substack{a,b\ge-2n-1\\a+b\ge-3n-2}}
E^{(n)}_{p,q}(a,b)v_a u_b z.
\end{equation}
In particular, every monomial in the deep triangle \(p+q\le n-1\)
is eliminated in favor of monomials on the diagonals
\(a+b\ge-3n-2\).  For fixed \(z\), all displayed sums are finite by lower truncation.
\end{lem}

\begin{proof}
For every strip~\({\rm(A)}\) seed, the product expansion
\eqref{eq:prod-C-expansion-mk}, together with
\eqref{eq:drop-comm-strip}, gives
\[
0\equiv_{U_A}
\sum_{a,b\ge-2n-1}
C^{(n)}_{n+p_0,0}(r;a,b)v_a u_b z.
\]
The coefficient matrix on the deep coordinates
\((a,b)=(-2n-1+p,-2n-1+q)\), \(p+q\le n-1\), is precisely
\(K_A^{(n)}\).

The matrix is invertible because the deep
triangle is exactly the grid \(\mathcal L_n\) in
\eqref{eq:formal-low-high-grids}.  In the proof of
Proposition~\ref{prop:AB-prodreg-full-rank}, the strip~\(\mathcal A\)
coefficient block on this grid is proved invertible.  The row change
\eqref{eq:formal-rho} stays inside the deep triangle because that
triangle is downward closed, and the column change
\eqref{eq:formal-hat-Gamma} stays inside each initial \(r\)-list.
Both changes are unitriangular.  Hence the original coefficient
matrix \(K_A^{(n)}\) is invertible as well.

Separating the deep coordinates from the remaining terms and solving
this square system gives \eqref{eq:Epq-direct-C-formula} and
\eqref{eq:deep-elim-generaln}.
\end{proof}

\begin{lem}[Strip C: diagonal elimination by the full triangular matrix]
\label{lem:generaln-stripC-deep-elim-explicit}
Fix \(n\ge1\), homogeneous \(u,v\in V\), and \(z\in V\).  On the
diagonal indexed by \(\eta\in\{0,\ldots,n-1\}\), designate the following
\(n-\eta\) pairs for elimination:
\begin{equation}\label{eq:stripC-elim-pair}
\begin{aligned}
(a_{\eta,\alpha},b_{\eta,\alpha})
:={}&\bigl(-2n-1+(2\eta+1+\alpha),\\
&\hspace{18mm}-2n-1+(n-\eta-1-\alpha)\bigr),
\qquad 0\le\alpha\le n-\eta-1.
\end{aligned}
\end{equation}
Let \(K_C^{(n)}\) be the square matrix whose rows are indexed by
\[
(m,r),\qquad 0\le m\le n-1,\quad 1\le r\le n-m,
\]
and whose columns are indexed by \((\eta,\alpha)\) as in
\eqref{eq:stripC-elim-pair}, with entries
\begin{equation}\label{eq:KC-direct-C-entry}
K_C^{(n)}\bigl((m,r),(\eta,\alpha)\bigr)
:=C^{(n)}_{m,0}
\bigl(r;a_{\eta,\alpha},b_{\eta,\alpha}\bigr).
\end{equation}
Then \(K_C^{(n)}\) is invertible.  For every pair \(a,b\ge-2n-1\)
which is not among the pairs \eqref{eq:stripC-elim-pair}, define
\begin{equation}\label{eq:Fka-full-inverse-formula}
\begin{aligned}
F^{(n)}_{\eta,\alpha}(a,b)
:={}&-
\sum_{m=0}^{n-1}\ \sum_{r=1}^{n-m}
\bigl(K_C^{(n)}\bigr)^{-1}
\bigl((\eta,\alpha),(m,r)\bigr)\\
&\hspace{35mm}\cdot C^{(n)}_{m,0}(r;a,b).
\end{aligned}
\end{equation}
Then
\begin{equation}\label{eq:stripC-elim-formula-general}
v_{a_{\eta,\alpha}}u_{b_{\eta,\alpha}}z
\equiv_{U_C}
\sum_{\substack{a,b\ge-2n-1\\(a,b)\text{ not eliminated}}}
F^{(n)}_{\eta,\alpha}(a,b)v_a u_b z.
\end{equation}
The only nonzero terms on the right occur among the \(2\eta+1\) surviving
pairs
\[
(a,b)=\bigl(-2n-1+p'',-2n-1+(n+\eta-p'')\bigr),
\qquad 0\le p''\le2\eta,
\]
on each strip-C diagonal, together with pairs satisfying
\(a+b\ge-2n-2\).  For fixed \(z\), all displayed sums are finite by lower truncation.  For a survivor on the same diagonal as the
eliminated pair, its coefficient has the closed form
\begin{equation}\label{eq:stripC-survivor-coeff}
\begin{aligned}
&F^{(n)}_{\eta,\alpha}
\bigl(-2n-1+p'',-2n-1+(n+\eta-p'')\bigr)\\
&\qquad=
(-1)^{\alpha+1}
\binom{2\eta+\alpha-p''}{\alpha}
\binom{n+\eta-p''}{n-\eta-1-\alpha}.
\end{aligned}
\end{equation}
\end{lem}

\begin{proof}
Each strip~\({\rm(C)}\) seed gives, modulo \(U_C\), the equation
\begin{equation}\label{eq:stripC-seed-expansion}
0\equiv_{U_C}
\sum_{a,b\ge-2n-1}C^{(n)}_{m,0}(r;a,b)v_a u_b z.
\end{equation}
The invertibility of \(K_C^{(n)}\) is a degree count: a seed with small
\(m+r\) cannot reach the deepest diagonals, which makes \(K_C^{(n)}\) block
triangular, and on each diagonal block only the lowest terms of the two
\(\Phi\)-factors survive, giving a Pascal-type matrix.

Partition its rows by \(s=m+r\in\{1,\ldots,n\}\).  Partition the
eliminated columns by \(s'=n-\eta\).  For a column on the diagonal indexed
by \(\eta\), consider the product
\[
\Phi_{a_{\eta,\alpha},1,v}(x)\Phi_{b_{\eta,\alpha},r,u}(x).
\]
Its lowest possible degree is
\[
3n+1-\eta-r.
\]
Consequently its coefficient in a row with \(m+r=s\) vanishes unless
\(\eta\ge n-s\), or equivalently unless \(s'\le s\).  With both block
indices ordered increasingly, \(K_C^{(n)}\) is block lower triangular.

On the diagonal block \(s'=s\), put \(\eta=n-s\) and
\(\mu=r-1=s-m-1\).  Only the lowest terms of the two \(\Phi\)-factors
can contribute to \eqref{eq:KC-direct-C-entry}, and
Lemma~\ref{lem:Phi-coeff-cutoff} gives the \(s\times s\) block
\begin{equation}\label{eq:stripC-Pascal-matrix}
M^{(n,\eta)}_{\mu,\alpha}
=\binom{n+\eta+1+\alpha}{\mu},
\qquad 0\le\mu,\alpha\le s-1.
\end{equation}
This is a binomial evaluation matrix at the consecutive nodes
\(n+\eta+1,\ldots,n+\eta+s\), and therefore
\begin{equation}\label{eq:stripC-Pascal-determinant}
\det M^{(n,\eta)}=1.
\end{equation}
Its inverse is
\begin{equation}\label{eq:stripC-inverse-formula}
\bigl(M^{(n,\eta)}\bigr)^{-1}_{\alpha,\mu}
=(-1)^{\alpha+\mu}
\sum_{j=\alpha}^{s-1}
\binom{j}{\alpha}
\binom{n+\eta+j-\mu}{j-\mu}.
\end{equation}
Indeed, Vandermonde convolution gives the factorization
\[
M^{(n,\eta)}=PQ,
\qquad
P_{\mu,j}=\binom{n+\eta+1}{\mu-j},
\qquad
Q_{j,\alpha}=\binom{\alpha}{j}.
\]
Here \(P\) is lower unitriangular and \(Q\) is upper unitriangular, which
also proves \eqref{eq:stripC-Pascal-determinant}.  Their standard binomial
inverses are
\[
(P^{-1})_{j,\mu}
=(-1)^{j-\mu}\binom{n+\eta+j-\mu}{j-\mu},
\qquad
(Q^{-1})_{\alpha,j}
=(-1)^{j-\alpha}\binom{j}{\alpha}.
\]
Multiplying \(Q^{-1}P^{-1}\) gives
\eqref{eq:stripC-inverse-formula}.

Every diagonal block is invertible, so the full block triangular matrix
\(K_C^{(n)}\) is invertible.  Solving all blocks simultaneously gives
the full-inverse formula \eqref{eq:Fka-full-inverse-formula}.

If the target is a survivor on the same diagonal \(\eta\), the lower-degree
vanishing just used shows that earlier blocks do not contribute.  Thus
only \eqref{eq:stripC-inverse-formula} is needed.  Substitution followed
by Vandermonde convolution reduces the inner sum to
\[
(-1)^j\binom{p''-2\eta-1}{j}
=\binom{2\eta-p''+j}{j}.
\]
The remaining sum is the hockey-stick identity, and gives
\eqref{eq:stripC-survivor-coeff}.  Finally, the same lowest-degree test
shows that strip~\({\rm(C)}\) has zero coefficient on the strip~A deep
triangle.  Hence the remaining possible support is exactly the support
stated in the lemma.
\end{proof}

The coefficients \eqref{eq:Epq-direct-C-formula} have a closed binomial form, which is
proved in Part~III \cite{Part3} together with its analogue for words of every length.

\section{Examples}
\label{sec:examples}

This section shows how the quadratic results are applied to a given vertex
operator algebra. Section~\ref{subsec:recipe} describes the procedure and the
resulting lists at levels one and two; the remaining subsections carry it out
for the Heisenberg, universal Virasoro, and affine $\mathfrak{sl}_2$ vertex
operator algebras, including the exact dimensions at levels one and two in the
first two cases.

\subsection{The procedure}
\label{subsec:recipe}
\label{subsec:general-S-low-level-eliminations}
\label{subsec:low-level-quadratic-examples}

Let $n\ge1$ and let $S$ be a finite homogeneous strong generating set of $V$
with $\ell$ elements, ordered once and for all.
\begin{enumerate}[label=\textup{(\arabic*)}]
\item \emph{Singular products.} For each ordered pair $(s,t)$ of elements of
$S$, compute $s_it$ for $i\ge0$. This determines the subspace
\eqref{eq:U-def} attached to the pair, which is contained in
$O_n^\circ(2,S,V)^+$ by \eqref{eq:pair-modulus-in-global-plus}. The vacuum
contributes nothing, since
$\mathbf1R^n_{m,k}z=[x^{m+2n+1}](1+x)^{n+k}z=0$ for $m\ge k\ge0$, and vacuum
descendants of a vector contribute nothing beyond that vector by
Lemma~\ref{lem:descendant-circle-subspace-inclusion}.
\item \emph{Pair lists.} For $s<t$, the $n(n+1)$ operators
$\mathrm{(A)}\cup\mathrm{(B)}$ of \eqref{eq:A-strip}--\eqref{eq:B-strip}, with
$(u,v)=(s,t)$, span $\mathcal Q_n((s,t))$ (Proposition~\ref{prop:empirical-rank}).
For $s=t$, the $2\lfloor(n+1)^2/4\rfloor$ odd-indexed members span
$\mathcal Q_n((s,s))$ (Proposition~\ref{prop:equal-input-odd-spanning}).
\item \emph{Global list.} By Corollary~\ref{cor:global-quadratic-bound}, these
lists together span $\mathcal Q_n(2,S)$, which gives the bound
\eqref{eq:global-quadratic-bound}.
\item \emph{Further relations.} Relations particular to $V$ can make listed
operators zero or dependent (Remark~\ref{rem:affine-level-one-rank-drop}).
Equality in the bounds requires a separate independence argument.
\end{enumerate}

At $n=1$, each pair $s<t$ contributes
\[
(s_{-1}t)R^1_{1,0}(\cdot),
\qquad
(s_{-1}t)R^1_{1,1}(\cdot),
\]
and each pair $(s,s)$ contributes the two corresponding operators
\[
(s_{-1}s)R^1_{1,0}(\cdot),
\qquad
(s_{-1}s)R^1_{1,1}(\cdot).
\]
Thus
\[
\dim\mathcal Q_1(2,S)\le\binom{\ell}{2}\cdot2+\ell\cdot2=\ell(\ell+1).
\]

At $n=2$, each pair $s<t$ contributes the six operators
\[
\begin{aligned}
&(s_{-1}t)R^2_{2,0}(\cdot),\quad
(s_{-2}t)R^2_{2,0}(\cdot),\quad
(s_{-1}t)R^2_{3,0}(\cdot),\\
&(s_{-1}t)R^2_{2,2}(\cdot),\quad
(s_{-2}t)R^2_{2,2}(\cdot),\quad
(s_{-1}t)R^2_{3,2}(\cdot),
\end{aligned}
\]
and each pair $(s,s)$ the four odd-indexed operators
\[
(s_{-1}s)R^2_{2,0}(\cdot),\quad
(s_{-1}s)R^2_{3,0}(\cdot),\quad
(s_{-1}s)R^2_{2,2}(\cdot),\quad
(s_{-1}s)R^2_{3,2}(\cdot).
\]
Thus
\[
\dim\mathcal Q_2(2,S)\le\binom{\ell}{2}\cdot6+\ell\cdot4=3\ell^2+\ell.
\]

\emph{Deep modes.} If one also wants the ordered products of deep modes in
terms of the others, Lemmas~\ref{lem:generaln-stripA-deep-elim-explicit}
and~\ref{lem:generaln-stripC-deep-elim-explicit} give them for each pair.
At level \(n=1\), each strip consists of one seed and eliminates one
ordered monomial.  This is the first consistency check on the formulas,
but the first nontrivial simultaneous systems occur at level \(n=2\).
There each strip has three seeds and eliminates three
ordered monomials.  Suppose that \(\wt u=\wt v\), and order the strip-A rows as
\[
(m,r)=(2,1),(2,2),(3,1)
\]
and its eliminated columns as
\[
(a,b)=(-5,-5),(-5,-4),(-4,-5).
\]
Direct substitution in \eqref{eq:KA-direct-C-entry} gives
\begin{equation}\label{eq:n2-equal-weight-KA}
K_A^{(2)}
=
\begin{pmatrix}
0&1&1\\
4&3\wt u+3&3\wt u+2\\
1&\wt u&\wt u
\end{pmatrix},
\qquad
\bigl(K_A^{(2)}\bigr)^{-1}
=
\begin{pmatrix}
-\wt u&0&1\\
\wt u-2&1&-4\\
3-\wt u&-1&4
\end{pmatrix}.
\end{equation}
Order the strip-C rows as
\[
(m,r)=(0,1),(1,1),(0,2)
\]
and its eliminated columns as
\[
(a,b)=(-2,-5),(-4,-4),(-3,-5).
\]
Then
\begin{equation}\label{eq:n2-equal-weight-KC}
K_C^{(2)}
=
\begin{pmatrix}
1&0&0\\
\wt u&1&1\\
3\wt u+2&3&4
\end{pmatrix},
\qquad
\bigl(K_C^{(2)}\bigr)^{-1}
=
\begin{pmatrix}
1&0&0\\
2-\wt u&4&-1\\
-2&-3&1
\end{pmatrix}.
\end{equation}
Equations \eqref{eq:Epq-direct-C-formula} and
\eqref{eq:Fka-full-inverse-formula}, together with
\eqref{eq:n2-equal-weight-KA} and \eqref{eq:n2-equal-weight-KC},
give all six \(n=2\) eliminations directly from the coefficients
\(C^{(2)}_{m,0}(r;a,b)\), with \(\wt u=1\) for weight-one currents and \(\wt u=2\) for the
conformal vector.
These matrices are used in the examples below. For unequal weights one uses the
defining entries \eqref{eq:KA-direct-C-entry} and \eqref{eq:KC-direct-C-entry}
without further change.

\subsection{The rank-one Heisenberg vertex operator algebra}
\label{subsec:heisenberg-low-level-eliminations}

Let \(M(1)\) be the rank-one Heisenberg vertex operator algebra and
\(S=\{h\}\), with \(h\) a normalized weight-one generator. As a vector space
\(M(1)\) is the polynomial algebra in the commuting modes
\(h_{-1},h_{-2},\ldots\) applied to \(\mathbf1\), and \(h_0=0\).

\emph{Singular products.} \(h_0h=0\), \(h_1h=\mathbf1\), and \(h_ih=0\) for
\(i\ge2\). Hence the subspace attached to \((h,h)\) is
\(U=O_n^\circ(h,M(1))\).

\emph{The quotient by \(U\).} By \eqref{eq:Rmk-def} with \(\wt h=1\),
\begin{equation}\label{eq:heisenberg-linear-relation-polynomial}
hR^n_{M,0}
=\sum_{j=0}^{n+1}\binom{n+1}{j}h_{-(M+2n+2-j)}
\qquad(M\ge0),
\end{equation}
which involves only negative modes and so acts by multiplication. By
\eqref{eq:Rmk-binomial-reduction}, \(U\) is the ideal generated by these
polynomials. The one with index \(M\) has leading mode \(h_{-(M+2n+2)}\) with
coefficient \(1\), so the modes \(h_{-j}\), \(j\ge2n+2\), are eliminated
recursively and
\begin{equation}\label{eq:heisenberg-linear-quotient-low-levels}
M(1)/O_n^\circ(h,M(1))\cong\C[h_{-1},\ldots,h_{-2n-1}]\mathbf1 .
\end{equation}

\emph{The quadratic operators.} Since \(\wt(h_{-1}h)=2\), \eqref{eq:Rmk-def}
gives
\[
(h_{-1}h)R^n_{m,k}=\sum_{i=0}^{n+k+2}\binom{n+k+2}{i}(h_{-1}h)_{i-m-2n-2},
\]
and \((h_{-1}h)_s=\sum_{r\in\Z}{:}h_rh_{s-1-r}{:}\) because
\(Y(h_{-1}h,x)={:}Y(h,x)Y(h,x){:}\). Put \(N_i=m+2n+3-i\ge2\). As \(h_0=0\),
\begin{equation}\label{eq:heisenberg-seed-vacuum-general}
(h_{-1}h)R^n_{m,k}v
=\sum_{i=0}^{n+k+2}\binom{n+k+2}{i}
\Biggl(\sum_{\substack{a,b\ge0\\a+b=N_i-2}}h_{-a-1}h_{-b-1}v
+2\sum_{q\ge1}h_{-N_i-q}h_qv\Biggr).
\end{equation}
By Vandermonde's identity
\(\binom{n+k+2}{i}=\sum_{b+j=i}\binom{k+1}{b}\binom{n+1}{j}\) and
\eqref{eq:heisenberg-linear-relation-polynomial}, the second part is
\begin{equation}\label{eq:heisenberg-positive-part-in-linear-relations}
2\sum_{q\ge1}\sum_{b=0}^{k+1}\binom{k+1}{b}\,
hR^n_{q+m+1-b,0}\bigl(h_qv\bigr)\in U,
\qquad q+m+1-b\ge q+m-k\ge1.
\end{equation}
Hence, for every \(v\in M(1)\),
\begin{equation}\label{eq:heisenberg-seed-arbitrary-vector-reduction}
(h_{-1}h)R^n_{m,k}v
\equiv_{U}
\Biggl(
\sum_{i=0}^{n+k+2}\binom{n+k+2}{i}
\sum_{\substack{a,b\ge0\\a+b=m+2n+1-i}}
h_{-a-1}h_{-b-1}
\Biggr)v,
\end{equation}
and reducing the polynomial in \eqref{eq:heisenberg-linear-quotient-low-levels}
gives the odd-indexed operators of Proposition~\ref{prop:equal-input-odd-spanning}
explicitly: at \(n=1\),
\begin{align}
(h_{-1}h)R^1_{1,0}v
&\equiv_{O_1^\circ(h,M(1))}
(h_{-3}+h_{-2})^2v,
\label{eq:heisenberg-n1-vacuum-A}\\
(h_{-1}h)R^1_{1,1}v
&\equiv_{O_1^\circ(h,M(1))}
(h_{-3}+2h_{-2}+h_{-1})^2v,
\label{eq:heisenberg-n1-vacuum-B}
\end{align}
and at \(n=2\)
\begin{align}
(h_{-1}h)R^2_{2,0}v
&\equiv_{O_2^\circ(h,M(1))}
2(h_{-4}+h_{-3})(h_{-5}+2h_{-4}+h_{-3})v,
\label{eq:heisenberg-n2-vacuum-A1}\\
(h_{-1}h)R^2_{3,0}v
&\equiv_{O_2^\circ(h,M(1))}
(h_{-5}-h_{-3})(h_{-5}+2h_{-4}+h_{-3})v,
\label{eq:heisenberg-n2-vacuum-A2}\\
(h_{-1}h)R^2_{2,2}v
&\equiv_{O_2^\circ(h,M(1))}
2(h_{-5}+3h_{-4}+3h_{-3}+h_{-2})
(h_{-4}+3h_{-3}+3h_{-2}+h_{-1})v,
\label{eq:heisenberg-n2-vacuum-B1}\\
(h_{-1}h)R^2_{3,2}v
&\equiv_{O_2^\circ(h,M(1))}
(h_{-5}+3h_{-4}+3h_{-3}+h_{-2})^2v.
\label{eq:heisenberg-n2-vacuum-B2}
\end{align}

\emph{Dimension.} By Proposition~\ref{prop:equal-input-odd-spanning} these
operators span \(\mathcal Q_1((h,h))\) and \(\mathcal Q_2((h,h))\). They act on
\eqref{eq:heisenberg-linear-quotient-low-levels} by multiplication by the
displayed polynomials, so applying them to \(v=\mathbf1\) shows that they are
independent exactly when the polynomials are. The two at \(n=1\) are, since
only the second involves \(h_{-1}\). At \(n=2\), the coefficient matrix of the
four polynomials on the monomials \(h_{-2}h_{-1}\), \(h_{-2}^2\),
\(h_{-3}^2\), \(h_{-4}h_{-3}\) is
\begin{equation}\label{eq:heisenberg-n2-independence-minor}
\begin{pmatrix}
0&0&2&6\\
0&0&-1&-2\\
2&6&18&24\\
0&1&9&18
\end{pmatrix},
\qquad
\det=4.
\end{equation}
Therefore
\begin{equation}\label{eq:heisenberg-low-level-exact-dimensions}
\boxed{
	\dim\mathcal Q_1((h,h))=2,
	\qquad
	\dim\mathcal Q_2((h,h))=4.}
\end{equation}

\emph{Eliminations.} At \(n=2\), set \(\wt u=\wt v=1\) in
\eqref{eq:n2-equal-weight-KA} and \eqref{eq:n2-equal-weight-KC}.
For example, solving the strip-A system for its deepest product gives
\begin{equation}\label{eq:n2-heisenberg-stripA}
\begin{aligned}
h_{-5}h_{-5}z\equiv_{U_A}
\bigl(&2h_{-5}h_{-3}+4h_{-4}h_{-4}
+8h_{-4}h_{-3}+3h_{-3}h_{-3}\bigr)z.
\end{aligned}
\end{equation}
Solve the full strip-C system.  By \([h_p,h_q]=p\delta_{p+q,0}\), negative modes commute;
in particular the eliminated monomial \(h_{-2}h_{-5}z\) equals \(h_{-5}h_{-2}z\), which also
occurs on the right-hand side.  Moving it to the left doubles its coefficient, and dividing
by \(2\) gives
\begin{equation}\label{eq:n2-heisenberg-stripC}
\begin{aligned}
h_{-2}h_{-5}z\equiv_{U_C}-\bigl(
&h_{-5}h_{-1}+h_{-4}h_{-3}+4h_{-4}h_{-2}
+3h_{-4}h_{-1}\\
&+2h_{-3}h_{-3}+6h_{-3}h_{-2}+3h_{-3}h_{-1}
+2h_{-2}h_{-2}+h_{-2}h_{-1}\bigr)z.
\end{aligned}
\end{equation}
These are two of the six \(n=2\) eliminations; the other four follow
from the same two displayed inverse matrices.

\subsection{The universal Virasoro vertex operator algebra}
\label{subsec:virasoro-low-level-eliminations}

Let \(V^{\mathrm{Vir}}(c,0)\) be the universal Virasoro vertex operator
algebra of central charge \(c\) and \(S=\{\omega\}\).

\emph{Singular products.} \(\omega_0\omega=L(-1)\omega=\omega_{-2}\mathbf1\),
\(\omega_1\omega=2\omega\), \(\omega_2\omega=0\),
\(\omega_3\omega=\frac c2\mathbf1\), and \(\omega_i\omega=0\) for \(i\ge4\).
By Lemma~\ref{lem:descendant-circle-subspace-inclusion} and the vanishing of
the vacuum slice, the subspace attached to \((\omega,\omega)\) is
\[
U=O_n^\circ\bigl(\omega,V^{\mathrm{Vir}}(c,0)\bigr),
\]
as in \cite[Example~5.2]{Part1}.

\emph{Lists.} The odd-indexed lists are those of the Heisenberg example with
\(h\) replaced by \(\omega\): two operators at \(n=1\) and the four operators
\((\omega_{-1}\omega)R^2_{m,k}\), \(m\in\{2,3\}\), \(k\in\{0,2\}\), at \(n=2\).
They are written out explicitly in \cite[Example~5.2]{Part1}.

\emph{Dimension.} By Proposition~\ref{prop:equal-input-odd-spanning} these
operators span \(\mathcal Q_1((\omega,\omega))\) and
\(\mathcal Q_2((\omega,\omega))\). Their independence modulo \(U\) is checked
on leading symbols. The one point needing care is that the leading symbols of
the elements of \(U\) form exactly the ideal generated by the leading symbols
of the length-one relations; otherwise a cancellation inside \(U\) could hide a
relation. Filter \(V^{\mathrm{Vir}}(c,0)\) by the
number of negative Virasoro modes.  Its associated graded space is the
polynomial algebra
\[
\operatorname{gr}V^{\mathrm{Vir}}(c,0)
\cong\C[x_2,x_3,\ldots],
\qquad
x_j=\sigma(L(-j)\mathbf1).
\]
For \(m\ge0\), the leading symbol of the length-one relation
\(\omega R^n_{m,0}(\cdot)\) is multiplication by
\begin{equation}\label{eq:virasoro-leading-linear-relation}
\sum_{i=0}^{n+2}\binom{n+2}{i}x_{m+2n+3-i}.
\end{equation}
Thus the corresponding leading-symbol ideal eliminates \(x_j\) for
\(j\ge5\) at \(n=1\) and for \(j\ge7\) at \(n=2\).
The two resulting polynomial quotients are therefore
\[
\C[x_2,x_3,x_4]
\qquad\text{and}\qquad
\C[x_2,x_3,x_4,x_5,x_6],
\]
respectively.  The leading symbols of the elements of \(U\) are exactly the elements of
these relation ideals.  Indeed, \(\omega R^n_{m,0}\) acts by the operator
\(\sum_{i=0}^{n+2}\binom{n+2}{i}L(i-m-2n-3)\), and by \cite[Lemma~2.3]{Part1} the values of
these operators, \(m\ge0\), on all of \(V^{\mathrm{Vir}}(c,0)\) span \(U\).  They lie in the
Lie algebra spanned by the \(L(j)\), \(j\le-2\), on which the central term vanishes.  Under
\(L(j)\mapsto-t^{j+1}\frac{d}{dt}\) the operator with index \(m\) corresponds to
\(-t^{-m-2n-2}(1+t)^{n+2}\frac{d}{dt}\), so their span corresponds to the vector fields
\(H(t)f(t^{-1})\frac{d}{dt}\), \(f\in\C[t^{-1}]\), with
\(H(t)=t^{-2n-2}(1+t)^{n+2}\).  This space is closed under the bracket:
\(\bigl[Hf\frac{d}{dt},Hg\frac{d}{dt}\bigr]=H\cdot H\bigl(f\frac{dg}{dt}-g\frac{df}{dt}\bigr)\frac{d}{dt}\),
and \(H\cdot t^{-2}\C[t^{-1}]=(1+t^{-1})^{n+2}t^{-n-2}\C[t^{-1}]\subseteq\C[t^{-1}]\).
Since \(V^{\mathrm{Vir}}(c,0)\) is free of rank one over the enveloping algebra of the
\(L(j)\), \(j\le-2\), a PBW basis whose first factors run over a basis of this subalgebra shows
that the associated graded space of \(U\) is the ideal generated by the symbols
\eqref{eq:virasoro-leading-linear-relation}.  Consequently, independence of leading
quadratic symbols in these smaller polynomial quotients implies independence modulo \(U\).

At \(n=1\), the coefficient matrix of the leading quadratic symbols of
the two odd-mode seeds on the monomials \(x_2^2,x_3x_2\) is
\[
\begin{pmatrix}
1&4\\
9&18
\end{pmatrix},
\qquad
\det=-18.
\]
At \(n=2\), the coefficient matrix of the leading quadratic symbols of
the four odd-mode residue operators in the \(n=2\) list of
\cite[Example~5.2]{Part1}, on
\[
x_5x_4,\qquad x_5^2,\qquad x_6x_5,\qquad x_6^2
\]
is
\begin{equation}\label{eq:virasoro-n2-independence-minor}
\begin{pmatrix}
18&6&2&0\\
0&3&4&1\\
44&8&2&0\\
48&16&8&1
\end{pmatrix},
\qquad
\det=4.
\end{equation}
Both matrices are invertible.  Therefore the two \(n=1\) and four
\(n=2\) vacuum values are linearly independent modulo \(U\).
Hence
\begin{equation}\label{eq:virasoro-low-level-exact-dimensions}
\boxed{
\dim\mathcal Q_1((\omega,\omega))=2,
\qquad
\dim\mathcal Q_2((\omega,\omega))=4}
\end{equation}
for \(V^{\mathrm{Vir}}(c,0)\). For a quotient of \(V^{\mathrm{Vir}}(c,0)\) in
which \(\omega\ne0\), the same operators span, but a null-vector ideal may impose
additional relations, so only \(\dim\mathcal Q_2((\omega,\omega))\le4\) is
asserted.

\emph{Eliminations and the redundant pair.} At \(n=2\) the general lemmas give
three strip-A and three strip-C eliminations. With equal inputs one in each
strip is redundant; we record all six and then show this directly.
The singular-iterate contributions lie in \(U\), and the commutator
contribution is zero by Theorem~\ref{thm:main-correction-results}\textup{(i)};
\(U_A\) and \(U_C\) are \eqref{eq:UA-diagonal-elimination} and
\eqref{eq:UC-diagonal-elimination} for \(u=v=\omega\). Fix arbitrary
\(z\in V^{\mathrm{Vir}}(c,0)\).

The three strip-A eliminations are
\begin{equation}\label{eq:n2-virasoro-stripA-1}
\begin{aligned}
\omega_{-5}\omega_{-5}z\equiv_{U_A}{}&\bigl(3\omega_{-5}\omega_{-3}+2\omega_{-5}\omega_{-2}+9\omega_{-4}\omega_{-4}\\
&\quad+18\omega_{-4}\omega_{-3}+9\omega_{-4}\omega_{-2}+3\omega_{-3}\omega_{-5}\\
&\quad+18\omega_{-3}\omega_{-4}+27\omega_{-3}\omega_{-3}+12\omega_{-3}\omega_{-2}\\
&\quad+2\omega_{-2}\omega_{-5}+9\omega_{-2}\omega_{-4}+12\omega_{-2}\omega_{-3}\\
&\quad+5\omega_{-2}\omega_{-2}\bigr)z.
\end{aligned}
\end{equation}
\begin{equation}\label{eq:n2-virasoro-stripA-2}
\begin{aligned}
\omega_{-5}\omega_{-4}z\equiv_{U_A}{}&\bigl(-2\omega_{-5}\omega_{-3}-\omega_{-5}\omega_{-2}-3\omega_{-4}\omega_{-4}\\
&\quad-6\omega_{-4}\omega_{-3}-3\omega_{-4}\omega_{-2}-3\omega_{-3}\omega_{-4}\\
&\quad-6\omega_{-3}\omega_{-3}-3\omega_{-3}\omega_{-2}-\omega_{-2}\omega_{-4}\\
&\quad-2\omega_{-2}\omega_{-3}-\omega_{-2}\omega_{-2}\bigr)z.
\end{aligned}
\end{equation}
\begin{equation}\label{eq:n2-virasoro-stripA-3}
\begin{aligned}
\omega_{-4}\omega_{-5}z\equiv_{U_A}{}&\bigl(-3\omega_{-4}\omega_{-4}-3\omega_{-4}\omega_{-3}-\omega_{-4}\omega_{-2}\\
&\quad-2\omega_{-3}\omega_{-5}-6\omega_{-3}\omega_{-4}-6\omega_{-3}\omega_{-3}\\
&\quad-2\omega_{-3}\omega_{-2}-\omega_{-2}\omega_{-5}-3\omega_{-2}\omega_{-4}\\
&\quad-3\omega_{-2}\omega_{-3}-\omega_{-2}\omega_{-2}\bigr)z.
\end{aligned}
\end{equation}
The three strip-C eliminations are
\begin{equation}\label{eq:n2-virasoro-stripC-1}
\begin{aligned}
\omega_{-2}\omega_{-5}z\equiv_{U_C}{}&\bigl(-\omega_{-5}\omega_{-2}-2\omega_{-5}\omega_{-1}-\omega_{-5}\omega_{0}\\
&\quad-\omega_{-4}\omega_{-3}-6\omega_{-4}\omega_{-2}-9\omega_{-4}\omega_{-1}\\
&\quad-4\omega_{-4}\omega_{0}-\omega_{-3}\omega_{-4}-6\omega_{-3}\omega_{-3}\\
&\quad-15\omega_{-3}\omega_{-2}-16\omega_{-3}\omega_{-1}-6\omega_{-3}\omega_{0}\\
&\quad-6\omega_{-2}\omega_{-4}-15\omega_{-2}\omega_{-3}-20\omega_{-2}\omega_{-2}\\
&\quad-14\omega_{-2}\omega_{-1}-4\omega_{-2}\omega_{0}-2\omega_{-1}\omega_{-5}\\
&\quad-9\omega_{-1}\omega_{-4}-16\omega_{-1}\omega_{-3}-14\omega_{-1}\omega_{-2}\\
&\quad-6\omega_{-1}\omega_{-1}-\omega_{-1}\omega_{0}-\omega_{0}\omega_{-5}\\
&\quad-4\omega_{0}\omega_{-4}-6\omega_{0}\omega_{-3}-4\omega_{0}\omega_{-2}\\
&\quad-\omega_{0}\omega_{-1}\bigr)z.
\end{aligned}
\end{equation}
\begin{equation}\label{eq:n2-virasoro-stripC-2}
\begin{aligned}
\omega_{-4}\omega_{-4}z\equiv_{U_C}{}&\bigl(-2\omega_{-5}\omega_{-3}-5\omega_{-5}\omega_{-2}-4\omega_{-5}\omega_{-1}\\
&\quad-\omega_{-5}\omega_{0}-10\omega_{-4}\omega_{-3}-21\omega_{-4}\omega_{-2}\\
&\quad-16\omega_{-4}\omega_{-1}-4\omega_{-4}\omega_{0}-3\omega_{-3}\omega_{-4}\\
&\quad-18\omega_{-3}\omega_{-3}-33\omega_{-3}\omega_{-2}-24\omega_{-3}\omega_{-1}\\
&\quad-6\omega_{-3}\omega_{0}-3\omega_{-2}\omega_{-4}-14\omega_{-2}\omega_{-3}\\
&\quad-23\omega_{-2}\omega_{-2}-16\omega_{-2}\omega_{-1}-4\omega_{-2}\omega_{0}\\
&\quad-\omega_{-1}\omega_{-4}-4\omega_{-1}\omega_{-3}-6\omega_{-1}\omega_{-2}\\
&\quad-4\omega_{-1}\omega_{-1}-\omega_{-1}\omega_{0}\bigr)z.
\end{aligned}
\end{equation}
\begin{equation}\label{eq:n2-virasoro-stripC-3}
\begin{aligned}
\omega_{-3}\omega_{-5}z\equiv_{U_C}{}&\bigl(\omega_{-5}\omega_{-3}+5\omega_{-5}\omega_{-2}+7\omega_{-5}\omega_{-1}\\
&\quad+3\omega_{-5}\omega_{0}+6\omega_{-4}\omega_{-3}+24\omega_{-4}\omega_{-2}\\
&\quad+30\omega_{-4}\omega_{-1}+12\omega_{-4}\omega_{0}-\omega_{-3}\omega_{-4}\\
&\quad+15\omega_{-3}\omega_{-3}+47\omega_{-3}\omega_{-2}+50\omega_{-3}\omega_{-1}\\
&\quad+18\omega_{-3}\omega_{0}+6\omega_{-2}\omega_{-4}+28\omega_{-2}\omega_{-3}\\
&\quad+50\omega_{-2}\omega_{-2}+40\omega_{-2}\omega_{-1}+12\omega_{-2}\omega_{0}\\
&\quad+3\omega_{-1}\omega_{-5}+15\omega_{-1}\omega_{-4}+30\omega_{-1}\omega_{-3}\\
&\quad+30\omega_{-1}\omega_{-2}+15\omega_{-1}\omega_{-1}+3\omega_{-1}\omega_{0}\\
&\quad+2\omega_{0}\omega_{-5}+8\omega_{0}\omega_{-4}+12\omega_{0}\omega_{-3}\\
&\quad+8\omega_{0}\omega_{-2}+2\omega_{0}\omega_{-1}\bigr)z.
\end{aligned}
\end{equation}
When \(u=v=\omega\), specializing the parity reduction
\eqref{eq:equal-input-parity-reduction} at \(n=2\) gives
\[
\begin{aligned}
2(\omega_{-2}\omega)R^2_{2,0}z
&\equiv_U(\omega_{-1}\omega)R^2_{2,0}z
 +8(\omega_{-1}\omega)R^2_{3,0}z,\\
2(\omega_{-2}\omega)R^2_{0,0}z
&\equiv_U-(\omega_{-1}\omega)R^2_{0,0}z
 +6(\omega_{-1}\omega)R^2_{1,0}z.
\end{aligned}
\]
Thus \(U_A\) is generated over \(U\) by its two \(r=1\) strip-A
operators, and \(U_C\) by its two \(r=1\) strip-C operators.
The same redundancy can be seen directly in the six eliminations.

Subtracting \eqref{eq:n2-virasoro-stripA-3} from
\eqref{eq:n2-virasoro-stripA-2} and commuting the reversed products
gives \(-\omega_{-10}z\) on the left and
\[
(4\omega_{-9}+6\omega_{-8}+4\omega_{-7}+\omega_{-6})z
\]
on the right.  Consequently the difference between the two complete
relations is
\[
-\bigl(\omega_{-10}+4\omega_{-9}+6\omega_{-8}
 +4\omega_{-7}+\omega_{-6}\bigr)z
=-\omega R^2_{4,0}z\equiv_U0.
\]
Hence \eqref{eq:n2-virasoro-stripA-3} adds no relation beyond
\eqref{eq:n2-virasoro-stripA-2} modulo \(U_A\).  Its left-hand side
is recovered by the exact identity
\[
\omega_{-4}\omega_{-5}z
=\omega_{-5}\omega_{-4}z+\omega_{-10}z.
\]

Likewise, after ordering the products, the left-hand side minus the
right-hand side of \eqref{eq:n2-virasoro-stripC-3} is
\(-5/2\) times the corresponding difference in
\eqref{eq:n2-virasoro-stripC-1}, plus
\[
2\omega R^2_{3,0}z+3\omega R^2_{2,0}z
 +2\omega R^2_{1,0}z+\tfrac12\omega R^2_{0,0}z.
\]
Each added term belongs to \(U\), so
\eqref{eq:n2-virasoro-stripC-3} adds no relation modulo \(U_C\).
Its left-hand side satisfies
\[
\omega_{-3}\omega_{-5}z
=\omega_{-5}\omega_{-3}z+2\omega_{-9}z.
\]
Thus the six formulas obtained from the general lemmas reduce to the
four eliminations \eqref{eq:n2-virasoro-stripA-1},
\eqref{eq:n2-virasoro-stripA-2}, \eqref{eq:n2-virasoro-stripC-1}, and
\eqref{eq:n2-virasoro-stripC-2}.  The complete differences vanish
modulo \(U\); the individual commutator terms need not.
All these identities hold for arbitrary \(z\).

\subsection{The affine \texorpdfstring{$\mathfrak{sl}_2$}{sl2} vertex operator algebra}
\label{subsec:affine-sl2-low-level-eliminations}

Let \(V=L_k(\mathfrak{sl}_2)\) at positive integral level \(k\), with
its minimal weight-one strong generating set \(S=\{e,h,f\}\), ordered
\(e<h<f\).

\emph{Singular products.} For \(x,y\in S\), the only possibly nonzero
singular products are \(x_0y=[x,y]\) and \(x_1y=k\langle x,y\rangle\mathbf1\).
They lie in the span of \(S\) or in the vacuum line, so every pair subspace is
contained in \(\sum_{s\in S}O_n^\circ(s,V)\subseteq O_n^\circ(2,S,V)^+\).

\emph{Lists and bound.} Section~\ref{subsec:recipe} with \(\ell=3\) gives
\[
\dim\mathcal Q_1(2,S)\le12,
\qquad
\dim\mathcal Q_2(2,S)\le30.
\]
All ordered pairs have equal weights, so the \(n=2\) eliminations use
\eqref{eq:n2-equal-weight-KA} and \eqref{eq:n2-equal-weight-KC} with
\(\wt u=\wt v=1\).

\emph{Level one.} The VOA \(L_1(\mathfrak{sl}_2)\) has the additional
relation \eqref{eq:affine-level-one-descendant-relation}, so
Remark~\ref{rem:affine-level-one-rank-drop} gives, for every \(n\ge1\),
\(\dim\mathcal Q_n((e,h))\le n(n-1)\). In particular,
\begin{equation}\label{eq:affine-eh-low-level-rank-drop}
\mathcal Q_1((e,h))=0,
\qquad
\dim\mathcal Q_2((e,h))\le2.
\end{equation}
The Chevalley involution
\(e\leftrightarrow f,\ h\mapsto-h\) sends
\(e_{-1}h=-e_{-2}\mathbf1\) to
\(f_{-1}h=f_{-2}\mathbf1\).  Since
\([h_{-1},f_{-1}]=-2f_{-2}\), it follows that
\(h_{-1}f=-f_{-2}\mathbf1\), and the same descendant argument gives
\begin{equation}\label{eq:affine-hf-low-level-rank-drop}
\mathcal Q_1((h,f))=0,
\qquad
\dim\mathcal Q_2((h,f))\le2.
\end{equation}
Integrability also gives \(e_{-1}e=0\) and \(f_{-1}f=0\). For \(n=1,2\) the
odd-indexed lists for \((e,e)\) and \((f,f)\) contain only residue operators
of these two vectors, so they contribute nothing. The remaining contributions
are bounded by
\[
\begin{array}{c|cc}
\text{pair}&n=1&n=2\\ \hline
(e,h),(h,f)&0&2\text{ each}\\
(e,f)&2&6\\
(h,h)&2&4.
\end{array}
\]
Consequently, for \(L_1(\mathfrak{sl}_2)\),
\begin{equation}\label{eq:sl2-level-one-stronger-low-level-bounds}
\boxed{
\dim\mathcal Q_1(2,S)\le4,
\qquad
\dim\mathcal Q_2(2,S)\le14.}
\end{equation}

\section{Concluding remarks}
\label{sec:conclusion}

The termwise correction theorem resolves the correction-residue
conjecture of Part~I at every word length and every level $n\ge0$.
The normalized diagonal kernel expresses each correction as a finite
sum of product residues of a contracted word with fewer $Y_L$ factors;
for the correction with $t=0$ these vanish, so the first $Y_L$ correction is
zero. The proof tracks every error as a shorter-word or
singular-word circle relation, without a degree hypothesis.

At positive level, the quadratic coefficient bases give a uniform
spanning set of $n(n+1)$ operators for each homogeneous ordered pair
modulo its length-one and singular-iterate relations. Equal inputs
reduce this list to its odd-indexed members. The formal coefficient
dimensions are exact; the actual operator bounds can be strict, as
the affine example shows. The two diagonal eliminations retain
explicit coefficients for reducing selected deep quadratic products.

Part~III \cite{Part3} uses these results to study longer product relations,
equality with the plus space beyond length $n+1$, and the finite-length
conjecture under its stated hypotheses. The original-letter tracking
of Theorem~\ref{thm:correction-residues-plus} is needed for that induction. Explicit recurrences for the
quadratic reduction coefficients are recorded in the ancillary file
\texttt{anc/\allowbreak explicit-\allowbreak quadratic-\allowbreak recurrences.pdf}.

\section*{Acknowledgments}
The author thanks Katrina Barron for discussions.

\section*{Code availability}
Exact-arithmetic Python scripts (standard library only) are supplied as
ancillary files with the arXiv source; \texttt{anc/README.txt} describes
what each script checks and the ranges tested. They check the normalized
diagonal kernel, the weight-removal identity and the spanning
coefficients, the transposition formula for $J$ (including negative
conformal weights), the strip eliminations, and the computations of
Section~\ref{sec:examples}. The ancillary file
\texttt{anc/\allowbreak explicit-\allowbreak quadratic-\allowbreak recurrences.pdf}
records explicit quadratic recurrences and $\Lambda$-triangularization
formulas. The proofs do not depend on these computations.

\section*{Declaration of generative AI and AI-assisted technologies in the manuscript preparation process}
During the preparation of this work the author used ChatGPT (OpenAI) and Claude (Anthropic) in order to assist with drafting and revising the manuscript, including its mathematical exposition, and to check identities by exact computer calculation. After using these tools, the author reviewed and edited the content as needed and takes full responsibility for the content of the published article.

\begingroup
\raggedright

\endgroup

\end{document}